\documentclass[a4paper,12pt]{amsart}

\usepackage{a4wide}
\usepackage{pgf,tikz}
\usepackage{amssymb}
\usepackage{enumerate}
	
\usepackage{amsmath}

\newif\ifdetails
\detailstrue
\newcommand{\DETAIL}[1]%
{\ifdetails\par\fbox{\begin{minipage}{0.9\linewidth}\textit{Detail:}
      #1\end{minipage}}\par\fi}
\newcommand{\TODO}[1]%
{\ifdetails\par\fbox{\begin{minipage}{0.9\linewidth}\textbf{TODO:}
      #1\end{minipage}}\par\fi}

\usepackage{makecell}
\usepackage{relsize}

\usepackage[pagewise,displaymath, mathlines]{lineno}

\newtheorem{lemma}{Lemma}
\newtheorem{proposition}[lemma]{Proposition}
\newtheorem{theorem}[lemma]{Theorem}

\theoremstyle{remark}

\DeclareMathOperator{\N}{N}

\DeclareMathOperator{\V}{V}

\usepackage{caption}
\usepackage{subcaption}

\usepackage{float} 
\usepackage[pdftex,a4paper,citecolor = blue, colorlinks=true,urlcolor=blue]{hyperref}
\usepackage{mathrsfs}
\usetikzlibrary{arrows}

\newcommand{\old}[1]{{}}

\title{Bicyclic graphs with the smallest and largest numbers of connected sets}

\author{Audace A. V. Dossou-Olory}
\thanks{}

\address{Audace A. V. Dossou-Olory \\ Institut National de l'Eau \\ and  Institut de Math\'ematiques et de Sciences Physiques, Dangbo \\ Universit\'e d'Abomey-Calavi, B\'enin \\ \newline \textbf{ORCID} \url{0000-0003-2065-117X}}
\email{audace@aims.ac.za} \email{audace.dossouolory@uac.bj}


\begin{document}

\begin{abstract}	
For a graph $G$ with vertex set $V$, let $\N(G)$ denote the number of nonempty subsets of $V$ that induce a connected graph in $G$. In this paper, we focus on determining $\N(G)$ for $G$ in the family $\mathbb{B}_n$ of $n$-vertex bicyclic graphs. We find in $\mathbb{B}_n$ the structures of those graphs that possess the smallest, the largest, as well as the second-largest values of $\N(G)$. Moreover, we compute the extreme values of $\N(G)$ over $ \mathbb{B}_n$.
\end{abstract}

\maketitle

\section{Introduction}\label{Intro:main}

All graphs considered in this paper are undirected, simple (no loop and no parallel edge), and of finite order (i.e. the number of vertices). A connected set in a graph $G=(V(G), E(G))$ is a nonempty set $S\subseteq V(G)$ that induces a connected subgraph in $G$. The edge set of the graph induced by $S$ consists of those edges of $G$ whose both endvertices belong to $S$. In this paper, we are interested in the overall number of connected sets in a connected graph $G$, henceforth denoted by $\N(G)$. Subgraphs constitute fundamental components of a graph, offering key insights into its intrinsic structure. For example, they can serve to identify essential building blocks and to uncover the functional characteristics of complex networks~\cite{Milo2002}.

In the past, much work has been devoted to determining the maximum and the minimum numbers of connected sets in different families of $n$-vertex graphs (i.e. graphs of order $n$), and also to characterising graphs attaining the extreme values. As a trivial example, precisely the complete graph achieves the maximum number of connected sets among all graphs on $n$ vertices. The literature contains a substantial number of research papers studying the extreme numbers of connected sets in acyclic graphs (i.e. in trees). As far as we know, this line of study on trees goes back to Sz\'ekely and Wang~\cite{Wang2005,Wang2005Plus}; trees with given degree sequence were considered in~\cite{AndriantianaV2021,Andriantiana2013,Zhang2015,ZhangZhang2013}; paper~\cite{Andriantiana2017} studies trees with given segment sequence; trees with given diameter appeared in~\cite{ZichongChen}; paper~\cite{AudDank2020} focuses on trees with given eccentric sequence; and many more, just to mention a few.

\medskip
We know that every $n$-vertex connected graph possesses $n-1$ edges or more, with the bound reached by trees only. A natural family of graphs after that of trees is perhaps the family of unicyclic graphs -- connected $n$-vertex graphs having exactly $n$ edges. In~\cite{Audacegenral2018}, the structures of $n$-vertex unicyclic graphs that reach the extreme numbers of connected sets are completed derived, among other things. One of the questions at the end of paper~\cite{Audacegenral2018} asks for constructive characterisation of the graphs that extremise the number of connected sets, over all connected graphs with prescribed order and number of cycles. In general, this
problem can appear out of reach due to the various ways in which the
cycles may intersect. However, special cases of so-called \textit{bicyclic
graphs} are still of interest. A bicyclic graph is a connected graph whose number of edges is exactly one more than its number of vertices. Such a graph necessarily contains two or three different cycles. Unicyclic and bicyclic graphs are among the most popular tree-like structures, and various graph invariants and parameters have been studied in these families.

In this paper, we shall determine closed formulas, as a function of $n$, for the smallest and largest possible numbers of connected sets in an $n$-vertex bicyclic graph, and further derive all extreme graph structures. Some known extremal results on trees and unicyclic graphs will play an important role in our current context of determining the number of connected sets of an $n$-vertex bicyclic graph. Our main underlying trick relies on certain graph transformations that modify the number of connected sets, while preserving both the numbers of vertices and edges.

\medskip
Many of the proofs in the present paper make use of results from works closely related to the topic. We therefore mention some of these works below.

We studied in~\cite{AudaceGirth2018} the two problems of maximising the number of connected sets of a unicyclic graph with prescribed order and girth, or prescribed order, girth, and number of pendant vertices. We determined in~\cite{AudaceCut2019} the maximum number of connected sets in a connected graph with given order and number of cut vertices, the minimum number among all cut vertex-free connected graphs, and the extreme numbers as a function of both order and number of pendant vertices. However, the result on the maximum, given $n$ vertices of which $p$ are pendant, first appeared in~\cite{Nordhauss} as a Nordhaus-Gaddum type. Paper~\cite{Audacegenral2018} extended~\cite{AudaceCut2019} to the maximum number of connected sets in unicyclic graphs with given order and number of cut vertices. In~\cite{AudTOUT}, we investigated the maximum of $\N(G)$ where $G$ is a unicyclic graph with $n$ vertices of which $c$ are cut vertices, and found that there are generally two extremal structures. For each of the above mentionned graph families, all extremal graph structures were characterised.

\medskip
The current paper strictly focuses on $n$-vertex bicyclic graphs. We define our first notation and state some initial results in Section~\ref{ch:Prelimi}; our main finding are covered in Sections~\ref{Sect:Extremal} and~\ref{Section:Case4TO7}.

\section{Preliminary} \label{ch:Prelimi}

Let $G=(V(G), E(G))$ be a connected graph. For $u,v \in V(G)$, we denote by $\N(G)_v$ and $\N(G)_{u,v}$ the number of connected sets in $G$ that involve $v$, and those that contain both $u$ and $v$, respectively. Let $G-v$ denote the graph that remains upon deleting vertex $v$ in $G$; more generally, $G-S$ represents the graph obtained from $G$ by deleting all elements of $S$.

A $u-v$ path in $G$ is a path between vertices $u$ and $v$, i.e. $(u=)w_0 w_1 \ldots w_{n-1} w_n(=v)$ such that $w_i w_{i+1} \in E(G)$ for all $0\leq i \leq n-1$. As usual, the $n$-vertex path will be denoted by $P_n$. A \emph{pendant vertex} (or a leaf) in $G$ is a vertex of degree $1$; an edge incident with a pendant vertex is called a \emph{pendant edge}.

By $C_n= v_0v_1v_2\ldots v_{n-1}$ we mean the cycle whose vertices are $v_0,v_1,v_2,\ldots, v_{n-1}$ in this order, i.e. $v_{n-1}$ is adjacent to $v_0$, and $v_jv_{j+1}$ is an edge for all $0\leq j \leq n-2$.  

\medskip
We mention some known results that are relevant to us.

\begin{lemma}[\cite{Wang2005}]\label{PnPnu}
For the path $P_n$ and $u$ a leaf of $P_n$, we have
\begin{align*}
\N(P_n)=\binom{n+1}{2} \quad \text{and} \quad \N(P_n)_u=n \,.
\end{align*}
\end{lemma}

\begin{lemma}[\cite{Audacegenral2018,AudaceGirth2018}]\label{cyclev}
For the cycle $C_n$ and $v\in V(C_n)$, we have
\begin{align*}
\N(C_n)_v=1+\binom{n}{2} \quad \text{and} \quad \N(C_n)=n^2-n+1\,.
\end{align*}
\end{lemma}

\medskip
At several occasions, we employ these two lemmas, without further reference.

A graph $G=(V(G), E(G))$ is said to be \emph{bicyclic} if $|E(G)|=|V(G)|+1$ and $G$ is connected. Every such graph must have at least four vertices. We denote by $\mathbb{B}_n$ the collection of all bicyclic graphs on $n$ vertices. Our goal is to determine both the smallest and the largest numbers of connected sets over $\mathbb{B}_n$, and to also characterise the extremal graphs.

\section{The extremal graph structures}\label{Sect:Extremal}

Bicyclic graphs are tree-like structures in the sense that they can be decomposed as a single `block' with trees attached to the vertices of this block.

\begin{lemma}\label{Gstar}
Let $G$ be a bicyclic graph. Then there is a unique maximal (w.r.t number of vertices) induced subgraph of $G$,  denoted by $G^*$, such that $G^*$ is also bicyclic but contains no pendant vertex.
\end{lemma}

\begin{proof}
If $G$ does not contain a pendant vertex, then $G=G^*$. Otherwise, let $u$ be a pendant vertex of $G$ and $H=G-u$. We have $|V(H)|=|V(G)|-1$ and $|E(H)|=|E(G)|-1$. Thus $H$ is also a bicyclic graph. If $H$ does not contain a pendant vertex, then $G^*=H$. Otherwise, $H$ contains a pendant vertex $v$ and $H-v$ is also a bicyclic graph. We can repeatedly delete a pendant vertex until we either obtain a bicyclic graph having no pendant vertex, or a bicyclic graph on four vertices. In the latter case, we get the complete graph on four vertices missing a single edge.

\medskip
In 2018, Malen~\cite{Malen2018} defined the {\it $k$-core} of a graph $G$ as the maximal (with respect to inclusion) subgraph, over all subgraphs of $G$ which have minimum degree at least $k$. In this paper, it was proved that the $k$-core of $G$ can be obtained by the process of deleting vertices with degree less than $k$ from $G$, along with all incident edges, one at a time until there are no vertices with degree less than $k$, and that regardless of the order in which vertices are deleted, this process terminates in the unique subgraph $k$-core (see also~\cite{Ellis2010}). Thus, for $k=2$, this gives precisely our graph $G^*$.
\end{proof}

As a first consequence of Lemma~\ref{Gstar}, we can construct any bicyclic graph $G$ from its subgraph $G^*$ by attaching trees rooted at the vertices of $G^*$. Given a bicyclic graph $G$, its subgraph $G^*$ can only be of the following types:
\begin{enumerate}
\item [{\bf Type~I}:] Two cycles with no common vertex. In this case, $G^*$ can be obtained from the disjoint union of $C_p,~P_r,~C_q$ by identifying one endvertex of $P_r$ with one vertex of $C_p$, and the other endvertex of $P_r$ (thus $r\geq 2$) with one vertex of $C_q$ (see Figure~\ref{fig:TypeI}). We denote this $G^*$ by $G[p,q,r]$.
\item [{\bf Type~II}:] Two cycles with only one common vertex. In this case, $G^*$ can be constructed from vertex disjoint cycles $C_p$ and $C_q$ by identifying one vertex in $C_p$ and one vertex in $C_q$ (see Figure~\ref{fig:TypeII}).
\item [{\bf Type~III}:] Two cycles with at least two common vertices. In this case, $G^*$ can be derived from vertex disjoint paths $P_a$, $P_b$, $P_c$ by first identifying one endvertex of each of these paths, and then identifying the other endvertices of these paths (see Figure~\ref{fig:TypeIII}).
\end{enumerate}

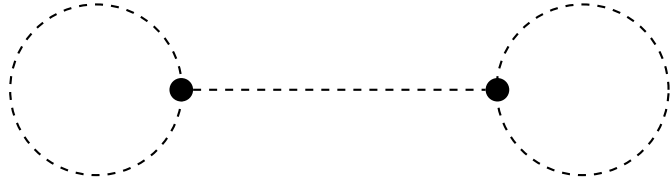
\begin{figure}[H]
	\begin{center}
		\begin{tikzpicture}
		[scale=0.4,inner sep=8mm, 
		vertex/.style={circle,thick,draw}, 
		thickedge/.style={line width=2pt}] 
		\node[vertex] (a1) at (0,0) [fill=white,dashed] {};
		\node[vertex] (a3) at (2.8,0) [inner sep=1mm,fill=black]{};
		\node[vertex] (a2) at (16,0) [fill=white,dashed]{};
		\node[vertex] (a4) at (13.2,0) [inner sep=1mm,fill=black]{};
		\draw[thick,black, dashed] (a3)--(a4);            
		\end{tikzpicture}
	\end{center}
	\caption{Type~I of $G^*$.} \label{fig:TypeI}
\end{figure}

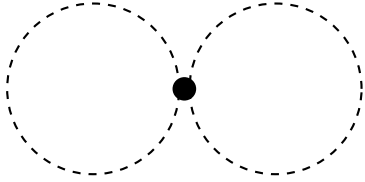
\begin{figure}[H]
	\begin{center}
		\begin{tikzpicture}
		[scale=0.4,inner sep=8mm, 
		vertex/.style={circle,thick,draw}, 
		thickedge/.style={line width=2pt}] 
		\node[vertex] (a1) at (0,0) [fill=white,dashed] {};
		\node[vertex] (a2) at (6,0) [fill=white,dashed] {};   
		\node[vertex] (a3) at (3,0) [inner sep=1mm,fill=black] {};           
		\end{tikzpicture}
	\end{center}
	\caption{Type~II of $G^*$.} \label{fig:TypeII}
\end{figure}

\begin{figure}[H]
	\begin{center}
		\begin{tikzpicture}
		[scale=0.4,inner sep=15.mm, 
		vertex/.style={circle,thick,draw}, 
		thickedge/.style={line width=2pt}] 
		\node[vertex] (a1) at (0,0) [fill=white,dashed] {};  
		\node[vertex] (a2) at (-5.3,0) [inner sep=1mm,fill=black] {};  
		\node[vertex] (a3) at (5.3,0) [inner sep=1mm,fill=black] {}; 
		\draw[thick,black, dashed] (a2)--(a3);      
		\end{tikzpicture}
	\end{center}
	\caption{Type~III of $G^*$.} \label{fig:TypeIII}
\end{figure}
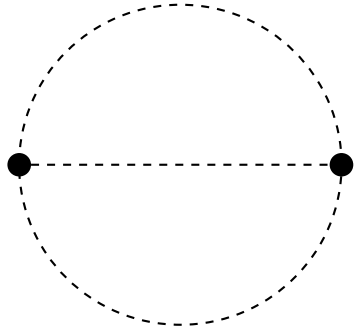

\medskip
Let us give a bit of relationship between Lemma~\ref{Gstar} and some relevant literature.

In their 2008 paper~\cite{BBBB}, Li, Chang, and Shiu studied the nullity (number of zero eigenvalues in the adjacency spectrum) of bicyclic graphs and characterised the $n$-vertex bicyclic graphs attaining the extremal values under consideration. In the course of their arguments, they stated, without proof, the three basic types of bicyclic graphs that are made explicit and systematically treated in Lemma~\ref{Gstar}. Liu et al.~\cite{CCCC} established sharp lower bounds for the Kirchhoff index of bicyclic graphs and characterised the extremal graphs attaining these bounds. Their work also considers structural configurations that overlap with one of the principal bicyclic configurations presented in Lemma~\ref{Gstar}. The 2023 paper~\cite{DDDD} by Khan et al. uses the basic bicyclic configurations in the study of the metric dimension/resolving sets of bicyclic graphs. The authors state, without a detailed proof, the corresponding base bicyclic graph and note its uniqueness. The recent work~\cite{EEEE} considers classes of bicyclic graphs according to their cycle arrangements and provides formulas for their Mostar index. The structural classification used there is closely related to the three types considered in Lemma~\ref{Gstar}.

\medskip
We continue with some general lemmas that will simplify the computation of the number of connected sets in certain categories of graphs.

\begin{lemma}\label{Lem:identify}
Let $H$ be obtained from the disjoint union of graphs $H_1$ and $H_2$ by identifying a vertex $u_1$ in $H_1$ and a vertex $u_2$ in $H_2$ (similar to Figure~\ref{fig:TypeII}). We have
\begin{align*}
\N(H)=\N(H_1)+\N(H_2) -1 + (\N(H_1)_{u_1} -1 )(\N(H_2)_{u_2} -1 )\,.
\end{align*}
\end{lemma}

\begin{proof}
Let $u$ be the only common vertex of $H_1$ and $H_2$ in $H$. We have
\begin{align*}
\N(H-u)=\N(H_1-u_1)+\N(H_2-u_2)=\N(H_1)-\N(H_1)_{u_1} + \N(H_2)-\N(H_2)_{u_2}
\end{align*}
and 
\begin{align*}
\N(H)_u=\N(H_1)_{u_1}\cdot \N(H_2)_{u_2}\,.
\end{align*}
From these two identities, we get
\begin{align*}
\N(H)&=\N(H-u)+\N(H)_u\\
&=\N(H_1)-\N(H_1)_{u_1} + \N(H_2)-\N(H_2)_{u_2}+\N(H_1)_{u_1}\cdot \N(H_2)_{u_2}\\
&=\N(H_1)+\N(H_2) -1 + (\N(H_1)_{u_1} -1 )(\N(H_2)_{u_2} -1 )\,,
\end{align*}
which proves the lemma.
\end{proof}

\medskip

\begin{lemma}\label{Lemm:vtovprime}
Let $H$ be a graph, $v$ a pendant vertex and $v'$ its neighbour in $H$. Then
\begin{align*}
\N(H)=\N(H-v)+1+\N(H-v)_{v'}\,.
\end{align*}
\end{lemma}

\begin{proof}
The graph $H$ decomposes as in Lemma~\ref{Lem:identify}, where $H_1=vv'$ and $H_2=H-v$. Thus 
\begin{align*}
\N(H)=3 +\N(H-v) -1 + (2 -1 )(\N(H-v)_{v'} -1 )=\N(H-v)+1+\N(H-v)_{v'}\,.
\end{align*}
\end{proof}
We employ Lemma~\ref{Lemm:vtovprime} with no further reference to it.

\medskip
For an integer $m\geq 4$, we define the {\it tadpole} graph $D_m$ to be the unicyclic graph obtained by identifying one vertex of the cycle $C_3$ with one leaf of the path $P_{m-2}$. We set $D_3:=C_3$. Furthermore, we define $L_n$ to be the graph $L_n:=G[3,3,n-4]$ when $n\geq 6$; see Figure~\ref{fig:TypeI}. We take two disjoint copies of $C_3$ and merge one vertex of each of them together to obtain the graph $L_5$ (a type~II graph from Figure~\ref{fig:TypeII}).

\begin{lemma}\label{Lem:FormLn}
We have 
\begin{align*}
\N(L_n)=\frac{(n+6)(n-1)}{2}
\end{align*}
for all $n\geq 5$.
\end{lemma}

\begin{proof}
It is easy to check that $\N(L_5)=22$. Let $n>5$, $v$ be the common vertex of a fixed $C_3$ and $P_{n-4}$ in $L_n$. The graph $L_n -v$ has $D_{n-3}$ as a connected component. It is proved in~\cite{Audacegenral2018} that 
\begin{align*}
\N(D_m)=\frac{(m-1)(m+4)}{2}
\end{align*}
for all $m\geq 3$. Thus
\begin{align*}
\N(L_n-v)=\N(P_2)+\N(D_{n-3})=3+ \frac{(n-4)(n+1)}{2}\,.
\end{align*}
Moreover, 
\begin{align*}
\N(L_n)_v=\N(C_3)_v \cdot \N(D_{n-2})_w\,,
\end{align*}
where $w$ is the unique pendant vertex of $D_{n-2}$. Since
\begin{align*}
\N(D_{n-2})_w=n-1\,, \quad \N(C_3)_v=4\,, 
\end{align*}
we deduce that $\N(L_n)_v=4(n-1)$ and that
\begin{align*}
\N(L_n)&=\N(L_n-v)+\N(L_n)_v\\
&=3+ \frac{(n-4)(n+1)}{2}+4(n-1)=\frac{(n+6)(n-1)}{2}\,,
\end{align*}
as needed.
\end{proof}

\medskip
Denote by $A_4$ the unique bicyclic graph on four vertices. For $n\geq 5$, we identify an endvertex of a $P_{n-3}$ with one vertex of degree two in $A_4$ to obtain $A_n \in \mathbb{B}_n$.

\begin{lemma}
We have
\begin{align*}
\N(A_n)=\frac{n^2+7n-16}{2}
\end{align*}
for all $n\geq 4$. In particular,
\begin{align*}
\N(L_5)=\N(A_5)=22\,, \quad \text{and} \quad \N(L_n)< \N(A_n) \quad \text{holds for all}~~~ n > 5\,.
\end{align*}

\end{lemma}
\begin{proof}
It is easy to check that $\N(A_4)=14$. For $n\geq 5$, denote by $u_1$ in $P_{n-3}$ and by $u_2$ in $A_4$ the two vertices that are identified to obtain $A_n$. We employ Lemma~\ref{Lem:identify} to get
\begin{align*}
\N(A_n)&=\N(P_{n-3})+\N(A_4) -1 + (\N(P_{n-3})_{u_1} -1 )(\N(A_4)_{u_2} -1 )\\
&=\binom{n-2}{2}+14-1+(n-3-1)(7-1)\\
&=\frac{(n-2)(n-3)}{2}+13+6(n-4)=\frac{n^2+7n-16}{2}\,.
\end{align*}
Now Lemma~\ref{Lem:FormLn} gives us
\begin{align*}
\N(A_n)-\N(L_n)=  \frac{n^2+7n-16}{2} -  \frac{(n+6)(n-1)}{2} = n-5 \geq 0
\end{align*}
for all $n\geq 5$, which proves the lemma.
\end{proof} 

\medskip
As a next step, we shall estimate the smallest number of connected sets containing a given vertex in an $n$-vertex bicyclic graph.

\begin{proposition}\label{Prop:nplus3}
If $G\in \mathbb{B}_n$ and $v\in V(G)$, then 
\begin{align*}
\N(G)_v \geq n+3\,.
\end{align*}
\end{proposition}

\begin{proof}
It is clear that $\N(A_4)_v=7$ if $v$ is of degree two, and  $\N(A_4)_v=8$ if $v$ is of degree three. So the proposition holds (with equality) for $n=4$. Henceforth, assume $n\geq 5$. We distinguish between three scenarios:
\begin{itemize}
\item $v$ is a pendant vertex of $G$, and $v'$ is the neighbour of $v$.\\
Since $G-v$ is also a bicyclic graph, we have
\begin{align*}
\N(G)_v=1+\N(G-v)_{v'} \geq 1+ (n-1+3)=n+3\,,
\end{align*}
where the inequality follows by induction on $n$.
\item $G$ contains a pendant vertex $u \neq v$.\\
Since $G-u$ is also a bicyclic graph, we have
\begin{align*}
\N(G)_v=\N(G-u)_v + \N(G)_{u,v} \geq (n-1+3)  + \N(G)_{u,v} \,,
\end{align*}
where the inequality follows by induction on $n$. The count for $\N(G)_{u,v}$ includes a shortest path between $u$ and $v$, thus $\N(G)_{u,v} \geq 1$. In particular,
\begin{align*}
\N(G)_v \geq  (n-1+3)  + 1 = n+3\,.
\end{align*}
\item $G$ does not have a pendant vertex.\\
All shortest paths starting at $v$ in $G$ contribute to the count $\N(G)_v$. Thus
$$\N(G)_v \geq n\,.$$
The degree of $v$ is at most $1$ in each of these paths.

Vertex $v$ together with any two of its neighbours also induces a connected graph in which $v$ is not pendant. We deduce the following:
\begin{itemize}
\item If $v$ has at least three neighbours in $G$, then $\N(G)_v \geq n +3 +1$, where the final $1$ accounts for $G$ itself. Vertex $v$ is of degree at least $2$ in each of these subgraphs.
\item If $v$ has only two neighbours, say $v_1,v_2$ in $G$, then either $v_1$ or $v_2$ must have a neighbour, say $w$ not belonging to $\{v,v_1,v_2\}$. Thus, each of the graphs induced by $\{v,v_1,v_2\}$ and $\{v,v_1,v_2,w\}$ is connected and none of them has $v$ as a pendant vertex. Hence  $\N(G)_v \geq n +2 +1$, where the final $1$ accounts for the whole $G$.
\end{itemize}
So such subgraphs containing $v$ and at least two of its neighbours in $G$ are all distinct and also different from the subpaths counted before.
\end{itemize}
This completes the proof of the proposition.
\end{proof}

\medskip
We list in Figure~\ref{fig:All5} all elements in the collection $\mathbb{B}_5$, alongside their numbers of connected sets, a well as the maxima of their vertex-contained connected sets (the details are omitted).
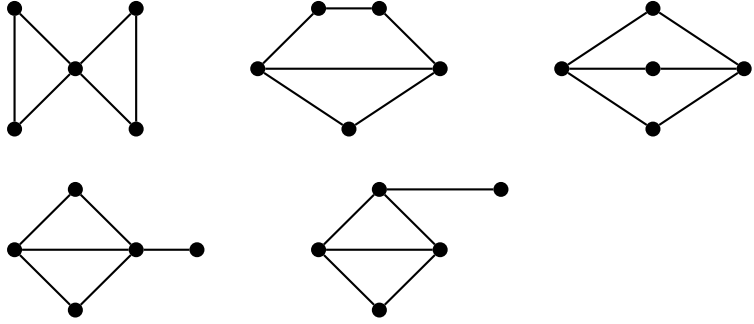
\begin{figure}[h]
	\begin{center}
		\begin{tikzpicture}
		[scale=0.4,inner sep=0.6mm, 
		vertex/.style={circle,thick,draw}, 
		thickedge/.style={line width=2pt}] 
		\node[vertex] (a1) at (0,2) [fill=black] {};  
		\node[vertex] (a2) at (0,-2) [fill=black] {};  
		\node[vertex] (a3) at (2,0) [fill=black] {}; 
		\node[vertex] (a4) at (4,2) [fill=black] {};  
		\node[vertex] (a5) at (4,-2) [fill=black] {}; 
		\draw[thick,black] (a1)--(a2) (a1)--(a3) (a2)--(a3) (a3)--(a4) (a3)--(a5) (a4)--(a5);      
		\node[vertex] (G) at (2,-2.3) [inner sep=1.mm,fill=white] {$L_5$};     
		\node[vertex] (b1) at (8,0) [fill=black] {};  
		\node[vertex] (b2) at (10,2) [fill=black] {};  
		\node[vertex] (b3) at (11,-2) [fill=black] {}; 
		\node[vertex] (b4) at (12,2) [fill=black] {};  
		\node[vertex] (b5) at (14,0) [fill=black] {}; 
		\draw[thick,black] (b1)--(b2) (b1)--(b3) (b1)--(b5) (b2)--(b4) (b3)--(b5) (b4)--(b5);      
		\node[vertex] (G) at (14,-2) [inner sep=1.mm,fill=white] {$E_{5,1}$};     
		\node[vertex] (c1) at (18,0) [fill=black] {};  
		\node[vertex] (c2) at (21,2) [fill=black] {};  
		\node[vertex] (c3) at (21,-2) [fill=black] {}; 
		\node[vertex] (c4) at (21,0) [fill=black] {};  
		\node[vertex] (c5) at (24,0) [fill=black] {}; 
		\draw[thick,black] (c1)--(c2) (c1)--(c3) (c1)--(c4) (c2)--(c5) (c3)--(c5) (c4)--(c5);
		\node[vertex] (G) at (24,-2.) [inner sep=1.mm,fill=white] {$E_{5,2}$};     
		\node[vertex] (d1) at (0,-6) [fill=black] {};  
		\node[vertex] (d2) at (2,-4) [fill=black] {};  
		\node[vertex] (d3) at (2,-8) [fill=black] {}; 
		\node[vertex] (d4) at (4,-6) [fill=black] {};  
		\node[vertex] (d5) at (6,-6) [fill=black] {}; 
		\draw[thick,black] (d1)--(d2) (d1)--(d3) (d1)--(d4) (d2)--(d4) (d3)--(d4) (d4)--(d5);
		\node[vertex] (G) at (4,-7.8) [inner sep=1.mm,fill=white] {$J_5$}; 
		\node[vertex] (e1) at (10,-6) [fill=black] {};  
		\node[vertex] (e2) at (12,-4) [fill=black] {};  
		\node[vertex] (e3) at (12,-8) [fill=black] {}; 
		\node[vertex] (e4) at (14,-6) [fill=black] {};  
		\node[vertex] (e5) at (16,-4) [fill=black] {}; 
	\draw[thick,black] (e1)--(e2) (e1)--(e3) (e1)--(e4) (e2)--(e4) (e3)--(e4) (e2)--(e5);      
	\node[vertex] (G) at (14.4,-7.8) [inner sep=1.mm,fill=white] {$A_5$};     
		\end{tikzpicture}
	\end{center}
	\caption{All bicyclic graphs on five vertices. Their numbers of connected sets are $22,~24,~26,~23,~22$, respectively. The vertex of degree $4$ uniquely realises the maximum among vertex-containing connected sets in both $L_5$ and $J_5$, while only vertices of degree $3$ attain the maximum in both $E_{5,1}$ and $E_{5,2}$. The quantity $\N(A_5)_t$ is maximum if and only if $t$ is the vertex adjacent to a leaf in $A_5$. The maxima for vertex-contained connected sets are given by $16,~14,~15,~16,~14$ in $L_5,~E_{5,1},~E_{5,2},~J_5,~A_5$, respectively.} \label{fig:All5}
\end{figure}

\medskip
A {\it cut vertex} in a graph $G$ is a vertex $u$ with the property that $G-u$ has more connected components than $G$.

\begin{theorem}[\cite{Audacegenral2018, AudaceCut2019}]\label{Thm:Audace2AND4}
The graph $D_n$ uniquely minimises the number of connected sets over the set of all $n$-vertex unicyclic graphs.

The cycle $C_n$ is the unique graph having the smallest number of connected sets among all $n$-vertex connected graphs with no cut vertex.
\end{theorem}

\medskip
We are now ready to state and prove our first theorem.

\begin{theorem}
Let $n>4$ and $G\in \mathbb{B}_n$. It holds that
$$\N(G) \geq \N(L_n)=\frac{(n+6)(n-1)}{2}
\,.$$ Equality happens in the case where $n>5$ if and only if $G$ is isomorphic to $L_n$. 
\end{theorem}

\begin{proof}
The formula for $\N(L_n)$ is given in Lemma~\ref{Lem:FormLn}. If $n=5$, then $\N(G) \geq 22$ by Figure~\ref{fig:All5}. Equality holds if and only if $G=L_5$ or $G=A_5$. Let $n>5$, and assume that $G$ has a pendant vertex, say $v$. Then $G-v \in \mathbb{B}_{n-1}$ and the induction hypothesis on $n$ yields
\begin{align*}
\N(G-v) \geq \frac{(n-1+6)(n-1-1)}{2}\,.
\end{align*}
This implies that
\begin{align*}
\N(G)=\N(G-v)+\N(G)_v \geq \frac{(n+5)(n-2)}{2}+ (n+3)\,,
\end{align*}
now using Proposition~\ref{Prop:nplus3}. Hence
\begin{align*}
\N(G)\geq \frac{(n+5)(n-2)}{2}+ (n+3) > \frac{(n+5)(n-2)}{2}+ (n+2)=\frac{(n+6)(n-1)}{2}\,,
\end{align*}
and we are done in this case. If $G$ does not have a pendant vertex, then $G=G^*$ is one of the three types in Figures~\ref{fig:TypeI},~\ref{fig:TypeII},~\ref{fig:TypeIII}. Henceforth, we assume that $G\in \mathbb{B}_n$ has the smallest number of connected sets. 

{\sc Claim~1}: $G$ cannot be of type II.\\
Suppose not. Replace the cycle $C_q$ ($q\geq p$) in $G=G^*$ with the tadpole graph $D_q$, as shown in Figure~\ref{fig:FromTypeII}, to obtain a graph $G'\in \mathbb{B}_n$. By $q\geq p$, we necessarily have $q\geq 4$.
\begin{figure}[h]
	\begin{center}
		\begin{tikzpicture}
		[scale=0.4,inner sep=8mm, 
		vertex/.style={circle,thick,draw}, 
		thickedge/.style={line width=2pt}] 
		\node[vertex] (a1) at (0,0) [fill=white,dashed] {};
		\node[vertex] (b1) at (10,0) [inner sep=1mm,fill=black] {};  
		\node[vertex] (b2) at (12,2) [inner sep=1mm,fill=black] {};   
		\node[vertex] (b3) at (12,-2) [inner sep=1mm,fill=black] {};   
		\node[vertex] (a3) at (3,0) [inner sep=1mm,fill=black] {};     
	\draw[thick,black, dashed] (a3)--(b1) ;   
	\draw[thick,black] (b1)--(b2) (b1)--(b3) (b2)--(b3);      
		\end{tikzpicture}
	\end{center}
	\caption{From Type~II to a Type~I.} \label{fig:FromTypeII}
\end{figure}
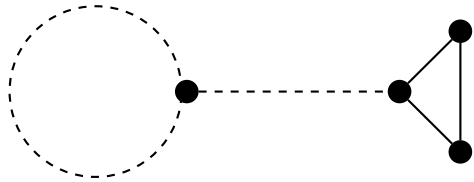
It follows from Lemma~\ref{Lem:identify} that
\begin{align*}
\N(G')-\N(G)=\N(D_q)-\N(C_q) + (\N(D_q)_{u_1} - \N(C_q)_{u_2} )(\N(C_p)_{u_2} -1 )\,,
\end{align*}
where $u_1$ is the unique pendant vertex in $D_q$. By Theorem~\ref{Thm:Audace2AND4}, $D_m$ is the unique $m$-vertex unicyclic graph having the minimum number of connected sets. Therefore,
\begin{align*}
\N(G')-\N(G) < (\N(D_q)_{u_1} - \N(C_q)_{u_2} )(\N(C_p)_{u_2} -1 )\,.
\end{align*}
Since $\N(D_q)_{u_1}=q+1$, while $\N(C_q)_{u_2}=1+\binom{q}{2}$, we deduce that
$$\N(D_q)_{u_1} - \N(C_q)_{u_2} < 0\,,$$
and that $\N(G')-\N(G) < 0$. This contradicts the minimality of $G$ with respect to the number of connected sets, thus proving the claim.

\medskip
{\sc Claim~2}: $G$ cannot be of type III.\\
Suppose not. By construction (see Figure~\ref{fig:TypeIII}), $G=G^*$ has no cut vertex (i.e. $G-u$ is still connected for every $u\in \V(G)$). By Theorem~\ref{Thm:Audace2AND4}, we then have $\N(G) > \N(C_n)$. On the other hand, 
\begin{align*}
\N(C_n) - \N(L_n) = n^2-n+1 - \frac{(n+6)(n-1)}{2} = \frac{n^2 -7n +8}{2}\,,
\end{align*}
which is positive for all $n>5$. Therefore, $\N(G) > \N(C_n) > \N(L_n)$. In particular, $G$ cannot reach the smallest number of connected sets over $\mathbb{B}_n$. This proves Claim~2.

\medskip
As a consequence of the above two Claims, a graph $G\in \mathbb{B}_n$ that minimises the number of connected sets must be of type~I.

{\sc Claim~3}: $G$ is isomorphic to $L_n$.\\
Suppose not. Replace the cycle $C_q$ ($q\geq p$) in $G=G^*$ (see Figure~\ref{fig:TypeI}) with the tadpole graph $D_q$ to obtain the same shape $G'$, as shown in Figure~\ref{fig:FromTypeII}. Using Lemma~\ref{Lem:identify}, we obtain
\begin{align*}
\N(G')-\N(G) = \N(D_q)-\N(C_q) + (\N(D_q)_{u_1} - \N(C_q)_{u_2} ) \big(\N (G - (V(C_q)\backslash \{u_2\}))_{u_2} -1 \big) \,,
\end{align*}
where $u_1$ is the unique pendant vertex of $D_q$. Since $\N(D_q)_{u_1} - \N(C_q)_{u_2} \leq 0$ and $D_m$ is the only $m$-vertex unicyclic graph with the smallest number of connected sets, we deduce that $\N(G')-\N(G) < 0$.

If $p=3$ then $G'$ is isomorphic to $L_n$. Otherwise, we apply the same graph transformation to $G'$ (replacing $C_p$ with $D_p$) to obtain $G''$. Hence  $\N(G'') < \N(G') < \N(G)$. Since $G''$ is isomorphic to $L_n$, the proof of the claim (thus also the theorem) is finished.
\end{proof}

\medskip
In what follows, we shift our focus on the maximisation problem. For our purposes, we shall even determine both the largest and the second-largest numbers of connected sets over $\mathbb{B}_n$, for all $n\geq 8$.

\medskip
Let $n\geq 6$ be an integer. Starting with two vertex disjoint cycles $C_3$ and $C_3$, first identify one vertex in a $C_3$ and one vertex in the other $C_3$ by vertex $w$, and then attach $n-5$ pendant edges to $w$. We denote the resulting graph by $R_n$ and set $R_5:=L_5$. Clearly, $R_n \in \mathbb{B}_n$ for all $n\geq 5$.

Similarly, we define $B_n$ to be the graph constructed from $A_4$ by attach $n-4$ pendant edges to the same vertex $w$ of degree $3$ in $A_4$. Clearly, $B_n \in \mathbb{B}_n$.

As usual, the $n$-vertex star graph is denoted by $S_n$.

\begin{lemma}\label{Form:RnBn}
We have
\begin{align*}
\N(R_n)&=n+1+2^{n-1}\\
\N(B_n)&=n+2+2^{n-1} =\N(R_n) +1
\end{align*}
for all $n\geq 5$.
\end{lemma}

\begin{proof}
We have $\N(R_5)=\N(L_5)=22$; so the formula is correct for $n=5$. Now assume $n>5$. The graph $R_n-w$ consists of two copies of $P_2$ and $n-5$ copies of $P_1$, while $w$ is adjacent to all other vertices of $R_n$. Thus
\begin{align*}
\N(R_n)=\N(R_n)_w + \N(R_n-w)=2^{n-1}+ 2\N(P_2)+ n-5=2^{n-1}+n+1\,.
\end{align*}
By Lemma~\ref{Lem:identify},
\begin{align*}
\N(B_n)=\N(A_4)+\N(S_{n-3}) -1 + (\N(A_4)_w -1 )(\N(S_{n-3})_w -1 )\,,
\end{align*}
where $w$ is also the center of the star. Thus
\begin{align*}
\N(B_n)=14 +(2^{n-4}+n-4) -1 + (8 -1 )(2^{n-4}-1 )=2^{n-1}+n+2\,,
\end{align*}
as stated.
\end{proof}

\medskip
We shall prove that $B_n$ uniquely maximises the number of connected sets among all $n$-vertex bicyclic graphs, provided $n>7$. Henceforth, we assume $n\geq 8$.

\medskip

Lemma~\ref{singleBranch} below is an important tool to proving our second theorem. It already appeared in~\cite{AudTOUT,AudaceGirth2018}.

\begin{lemma}[\cite{AudTOUT,AudaceGirth2018}]\label{singleBranch}
	Let $L,M,R$ be three non-trivial connected graphs whose vertex sets are pairwise disjoints. Let $l \in V(L),~ r\in V(R)$ and $u,v \in V(M)$ be fixed vertices such that $u \neq v$. Denote by $G$ the graph obtained from $L,M,R$ by identifying $l$ with $u$, and $r$ with $v$. Similarly, let $G'$ (resp. $G''$) be the graph obtained from $L,M,R$ by identifying both $l,r$ with $u$ (resp. both $l,r$ with $v$); see Figure~\ref{diagGGpGpp} for a diagram of these graphs.
	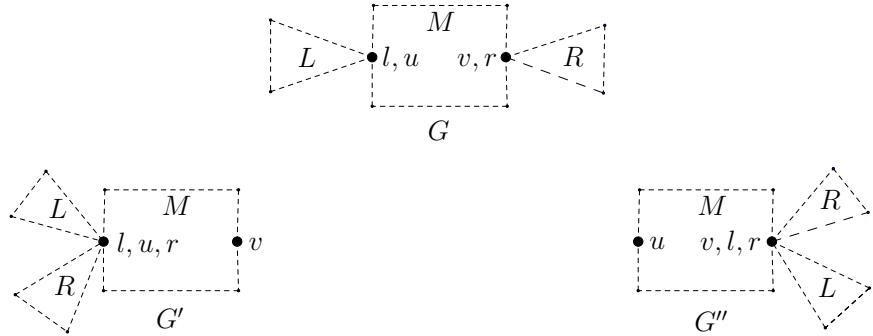
\begin{figure}[htbp]\centering
		\definecolor{qqqqff}{rgb}{0.,0.,1.}
		\resizebox{0.7\textwidth}{!}{%
		\begin{tikzpicture}[line cap=round,line join=round,>=triangle 45,x=1.0cm,y=1.0cm]
		rectangle (17.25469285535401,15.543118314192673);
		\draw [dash pattern=on 2pt off 2pt] (8.02,12.78)-- (10.02,12.78);
		\draw [dash pattern=on 2pt off 2pt] (10.02,12.78)-- (10.,12.);
		\draw [dash pattern=on 2pt off 2pt] (10.,12.)-- (10.02,11.26);
		\draw [dash pattern=on 2pt off 2pt] (10.02,11.26)-- (8.,11.26);
		\draw [dash pattern=on 2pt off 2pt] (8.,11.26)-- (8.,12.);
		\draw [dash pattern=on 2pt off 2pt] (8.,12.)-- (8.02,12.78);
		\draw [dash pattern=on 2pt off 2pt] (6.48,12.56)-- (8.,12.);
		\draw [dash pattern=on 2pt off 2pt] (6.48,11.48)-- (8.,12.);
		\draw [dash pattern=on 2pt off 2pt] (6.48,12.56)-- (6.48,11.48);
		\draw [dash pattern=on 2pt off 2pt] (10.,12.)-- (11.48,12.48);
		\draw [dash pattern=on 4pt off 4pt] (10.,12.)-- (11.46,11.46);
		\draw [dash pattern=on 2pt off 2pt] (11.48,12.48)-- (11.46,11.46);
		\draw (8.002875389073733,12.308886980787792) node[anchor=north west] {$l,u$};
		\draw (9.126357554291436,12.19457622267538) node[anchor=north west] {$v,r$};
		\draw (6.736460949518799,12.287683803799489) node[anchor=north west] {$L$};
		\draw (10.70036563986977,12.304751529462257) node[anchor=north west] {$R$};
		\draw [dash pattern=on 2pt off 2pt] (4.,10.)-- (6.,10.);
		\draw [dash pattern=on 2pt off 2pt] (6.,10.)-- (5.98,9.22);
		\draw [dash pattern=on 2pt off 2pt] (5.98,9.22)-- (6.,8.48);
		\draw [dash pattern=on 2pt off 2pt] (6.,8.48)-- (3.98,8.48);
		\draw [dash pattern=on 2pt off 2pt] (3.98,8.48)-- (3.98,9.22);
		\draw [dash pattern=on 2pt off 2pt] (3.98,9.22)-- (4.,10.);
		\draw [dash pattern=on 2pt off 2pt] (3.12,10.28)-- (3.98,9.22);
		\draw [dash pattern=on 2pt off 2pt] (2.6,9.6)-- (3.98,9.22);
		\draw [dash pattern=on 2pt off 2pt] (3.12,10.28)-- (2.6,9.6);
		\draw (4.031552359486242,9.502599906980058) node[anchor=north west] {$l,u,r$};
		\draw (6.006612285785838,9.41691066509247) node[anchor=north west] {$v$};
		\draw (3.008420807842294,10.0050302314765) node[anchor=north west] {$L$};
		\draw (3.087042324067116,8.845683517584332) node[anchor=north west] {$R$};
		\draw [dash pattern=on 2pt off 2pt] (12.,10.)-- (14.,10.);
		\draw [dash pattern=on 2pt off 2pt] (14.,10.)-- (13.98,9.22);
		\draw [dash pattern=on 2pt off 2pt] (13.98,9.22)-- (14.,8.48);
		\draw [dash pattern=on 2pt off 2pt] (14.,8.48)-- (11.98,8.48);
		\draw [dash pattern=on 2pt off 2pt] (11.98,8.48)-- (11.98,9.22);
		\draw [dash pattern=on 2pt off 2pt] (11.98,9.22)-- (12.,10.);
		\draw [dash pattern=on 2pt off 2pt] (13.98,9.22)-- (14.9,10.32);
		\draw [dash pattern=on 4pt off 4pt] (13.98,9.22)-- (15.42,9.66);
		\draw [dash pattern=on 2pt off 2pt] (14.9,10.32)-- (15.42,9.66);
		\draw (12.000266144323993,9.432599906980058) node[anchor=north west] {$u$};
		\draw (12.761465421630708,9.532775213766935) node[anchor=north west] {$v,l,r$};
		\draw (14.542891846445562,10.147962505813732) node[anchor=north west] {$R$};
		\draw [dash pattern=on 2pt off 2pt] (2.66,8.42)-- (3.98,9.22);
		\draw [dash pattern=on 2pt off 2pt] (3.44,7.86)-- (3.98,9.22);
		\draw [dash pattern=on 2pt off 2pt] (2.66,8.42)-- (3.44,7.86);
		\draw (8.654102536581874,12.838560337733874) node[anchor=north west] {$M$};
		\draw (4.70415799076956,10.0550302314765) node[anchor=north west] {$M$};
		\draw (12.718182533719721,10.0550302314765) node[anchor=north west] {$M$};
		\draw [dash pattern=on 2pt off 2pt] (13.98,9.22)-- (15.44,8.52);
		\draw [dash pattern=on 2pt off 2pt] (14.78,7.92)-- (13.98,9.22);
		\draw [dash pattern=on 2pt off 2pt] (15.44,8.52)-- (14.78,7.92);
		\draw [dash pattern=on 2pt off 2pt] (15.44,8.52)-- (14.78,7.92);
		\draw [dash pattern=on 2pt off 2pt] (15.44,8.52)-- (14.78,7.92);
		\draw (14.521338055883508,8.817237308146387) node[anchor=north west] {$L$};
		\draw (8.682548746019819,11.18282315480661) node[anchor=north west] {$G$};
		\draw (4.618293442095093,8.37342849987477) node[anchor=north west] {$G'$};
		\draw (12.675250259382487,8.330496225537535) node[anchor=north west] {$G''$};
		
		\draw [fill=black] (8.02,12.78) circle (0.5pt);
		\draw [fill=black] (10.02,12.78) circle (0.5pt);
		\draw [fill=black] (8.,11.26) circle (0.5pt);
		\draw [fill=black] (10.02,11.26) circle (0.5pt);
		\draw [fill=black] (8.,12.) circle (2.0pt);
		\draw [fill=black] (10.,12.) circle (2.0pt);
		\draw [fill=black] (6.48,12.56) circle (0.5pt);
		\draw [fill=black] (6.48,11.48) circle (0.5pt);
		\draw [fill=qqqqff] (11.48,12.48) circle (0.5pt);
		\draw [fill=qqqqff] (11.46,11.46) circle (0.5pt);
		\draw [fill=black] (4.,10.) circle (0.5pt);
		\draw [fill=black] (6.,10.) circle (0.5pt);
		\draw [fill=black] (3.98,8.48) circle (0.5pt);
		\draw [fill=black] (6.,8.48) circle (0.5pt);
		\draw [fill=black] (3.98,9.22) circle (2.0pt);
		\draw [fill=black] (5.98,9.22) circle (2.0pt);
		\draw [fill=black] (3.12,10.28) circle (0.5pt);
		\draw [fill=black] (2.6,9.6) circle (0.5pt);
		\draw [fill=black] (12.,10.) circle (0.5pt);
		\draw [fill=black] (14.,10.) circle (0.5pt);
		\draw [fill=black] (11.98,8.48) circle (0.5pt);
		\draw [fill=black] (14.,8.48) circle (0.5pt);
		\draw [fill=black] (11.98,9.22) circle (2.0pt);
		\draw [fill=black] (13.98,9.22) circle (2.0pt);
		\draw [fill=black] (15.44,8.52) circle (0.5pt);
		\draw [fill=black] (14.78,7.92) circle (0.5pt);
		\draw [fill=qqqqff] (14.9,10.32) circle (0.5pt);
		\draw [fill=qqqqff] (15.42,9.66) circle (0.5pt);
		\draw [fill=black] (3.44,7.86) circle (0.5pt);
		\draw [fill=black] (2.66,8.42) circle (0.5pt);
		\end{tikzpicture}}
		\caption{The graphs $G,G',G''$ constructed in Lemma~\ref{singleBranch}.}\label{diagGGpGpp}
	\end{figure}

We have 
\begin{align*}
	\N(G') - \N(G)&= (\N(R)_r -1) (\N(L)_l \cdot \N(M-v)_u - \N(M-u)_v)\,,\\
	\N(G'') - \N(G)&= (\N(L)_l -1) (\N(R)_r \cdot \N(M-u)_v - \N(M-v)_u)\,,
	\end{align*}
	In particular, it holds that
	\begin{align*}
	\N(G') > \N(G) ~~ \text{or} ~~ \N(G'') > \N(G)\,.
	\end{align*}
\end{lemma}

\medskip
The next proposition is fairly well-known, see for instance~\cite{Audacegenral2018,Wang2005}.

\begin{proposition}\label{upperTr}
Let $T$ be a tree on $n$ vertices whose root is $v$. Then $\N(T)_v\leq 2^{n-1}$. Equality holds if and only if $T$ is the star $S_n$ and $v$ is the center of $T$.
\end{proposition}

\medskip
Let $n>2$ and $Q_n$ be the graph that results from adding one edge between two leaves of the star $S_n$. Clearly, $Q_n$ contains only one cycle, which is $C_3$. 

\begin{theorem}[\cite{Audacegenral2018}]\label{Thm:MaxQn}
The graph $Q_n$ uniquely maximises the number of connected sets among all unicyclic graphs of order $n>5$. 
\end{theorem}

\medskip
The (unique) largest and a second-largest bicyclic graphs with respect to the number of connected sets are described in the next theorem.

\begin{theorem}\label{Thm:MainOnMax}
Let $n\geq 8$ and $G \in \mathbb{B}_n$. Then
$$\N(G) \leq \N(B_n)\,,$$
with equality if and only if $G$ is isomorphic to $B_n$. Moreover,
$$\N(G) \leq \N(R_n)$$ holds for all $G\in \mathbb{B}_n \backslash \{B_n\}$.
\end{theorem}

\begin{proof}
Let $n\geq 8$ and $G_0 \in \mathbb{B}_n$. By Lemma~\ref{Gstar}, $G_0$ can be obtained from its bicyclic subgraph $G_0^*$ by attaching trees at the vertices of $G_0^*$. We make the specialisation $M:=G_0^*$ and repeatedly apply Lemma~\ref{singleBranch}, until we get a new graph $G_1\in \mathbb{B}_n$. More precisely, starting from $G_0$, we always choose the reattachment in Lemma~\ref{singleBranch} that strictly increases the number of connected sets, until the `dangling trees' in $G_0$ are all rooted at a single vertex of $G^*_0$. The resulting graph is $G_1$ and $G^*_1=G^*_0$. Therefore, $G_1$ can be constructed from $G_1^*=G_0^*$ by attaching (possibly none) a tree $T_1$ rooted at only one vertex, say $v$ of $G_1^*$. Hence $\N(G_1)\geq \N(G_0)$, with equality if and only if both $G_1$ is isomorphic to $G_0$ and $G_0$ already consisted of $G^*_0$ with a single tree rooted at one vertex of $G^*_0$.

\medskip
Consider the bicyclic graph $G_1$. As was described in Lemma~\ref{Lem:identify}, `replace' the tree $T_1$ with the star $S_{|V(T_1)|}$ whose center is $v$, to obtain a graph $G_2\in \mathbb{B}_n$. This gives us
\begin{align*}
\N(G_2)-\N(G_1)=\N(S_{|V(T_1)|}) - \N(T_1) + (\N(S_{|V(T_1)|})_{v} - \N(T_1)_v)(\N(G_1^*)_{v} -1 )\,.
\end{align*}
It is a well known statement (see, as eg.~\cite{Audacegenral2018,Wang2005}) that $S_m$ uniquely reaches the maximum number of connected sets among all $m$-vertex trees. Together with Proposition~\ref{upperTr}, we deduce that $\N(G_2) \geq \N(G_1)$, with equality if and only if $T_1$ is a star and $v$ its center.

Consider $G_2$, and let us first focus on its subgraph $G_2^*=G_1^*=G_0^*$.

\medskip

{\sc Case~1}. Assume $G_2^*$ is of type~I or type~II (see Figures~\ref{fig:TypeI} and~\ref{fig:TypeII}, with $q\geq p$), and let $H:=G_2^* - (V(C_p) \backslash u)$, where $u$ is the unique common vertex of $C_p$ and $H$ in $G_2^*$.

\begin{itemize}
\item As long as $|V(H)|\geq 5$, `replace' the part $H$ of $G_2^*$ with the graph $Q_{|V(H)|}$ to obtain a graph $G'_3\in \mathbb{B}_n$, as shown in Figure~\ref{fig:TypeIorIItoQ}.
\begin{figure}[h]
	\begin{center}
		\begin{tikzpicture}
		[scale=0.4,inner sep=8mm, 
		vertex/.style={circle,thick,draw}, 
		thickedge/.style={line width=2pt}] 
		\node[vertex] (a1) at (0,0) [fill=white,dashed] {};
		\node[vertex] (b1) at (3,0) [inner sep=1mm,fill=black] {};  
		\node[vertex] (b2) at (9,1.5) [inner sep=1mm,fill=black] {};   
		\node[vertex] (b3) at (9,-1.5) [inner sep=1mm,fill=black] {};   
	\node[vertex] (c1) at (3,-3) [inner sep=1mm,fill=black] {}; 
	\node[vertex] (d) at (4,-3) [inner sep=0.4mm,fill=white] {}; 
		\node[vertex] (e) at (5,-3) [inner sep=0.4mm,fill=white] {}; 
		\node[vertex] (f) at (6,-3) [inner sep=0.4mm,fill=white] {}; 
	\node[vertex] (c2) at (7,-3) [inner sep=1mm,fill=black] {};  
	\draw[thick,black] (b1)--(b2) (b1)--(b3) (b2)--(b3);     
	\draw[thick,black] (b1)--(c1) (b1)--(c2);      
		\end{tikzpicture}
	\end{center}
	\caption{Replacing part $H$ of Type~I or Type~II with $Q_{|V(H)|}$ to obtain $G'_3\in \mathbb{B}_n$. This map sends $G^*_2$ to $G'_3$.} \label{fig:TypeIorIItoQ}
\end{figure}
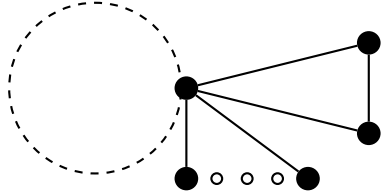

Vertex $u$ is adjacent to all other vertices in $Q_{|V(H)|}$, thus $\N(Q_{|V(H)|})_u=2^{|V(H)| -1}$. Hence Lemma~\ref{Lem:identify} combined with Theorem~\ref{Thm:MaxQn} imply that
\begin{align*}
\N(G'_3) - \N(G^*_2) &= \N(Q_{|V(H)|}) - \N(H) + (\N(C_p)-1)(\N(Q_{|V(H)|})_u - \N(H)_u)\\
& \geq (\N(C_p)-1)(\N(Q_{|V(H)|})_u - \N(H)_u) \geq \N(Q_{|V(H)|})_u - \N(H)_u \,.
\end{align*}
Vertex $u$ cannot be adjacent to all other vertices of $H$ in $G^*_2$, unless $q=3$ and $G^*_2$ is of type~II (in this case $|V(H)|=q=3 <5$, a contradiction).

Hence, $u$ is not adjacent to all other vertices of $H$. Thus
\begin{align*}
\N(G'_3) - \N(G^*_2) \geq \N(Q_{|V(H)|})_u - \N(H)_u > 0\,,
\end{align*}
that is $\N(G'_3) > \N(G^*_2)$ in the case where $|V(H)| \geq 5$.

\medskip
{\sc Subcase~1}. $p\geq 5$.

Whenever necessary, we can apply this same transformation on the part $C_p$ of $G'_3$ to obtain the graph $R_{|V(G'_3)|}$ and the inequality $\N(R_{|V(G'_3)|}) > \N(G'_3)$. Hence 
\begin{align*}
\N(R_{|V(G'_3)|}) > \N(G'_3) > \N(G^*_2)\,.
\end{align*}

{\sc Subcase~2}. $|V(H)| \geq 5$ and $p\leq 4$.

Recall that $|V(G'_3)|=|V(H)| + p - 1$. If $p=3$, then $G'_3$ is isomorphic to $R_{|V(G'_3)|}$. Otherwise, $p=4$. In this situation, using Lemma~\ref{Lem:identify} we can compare directly $\N(G'_3)$ with $\N(R_{|V(G'_3)|})$ in the following manner:
\begin{align*}
\N(R_{|V(G'_3)|}) - \N(G'_3) &= \N(D_4) - \N(C_4) + (\N(Q_{|V(G'_3)|-3})_z -1) ( \N(D_4)_x - \N(C_4)_y) \\
& = 12-13 + (\N(Q_{|V(H)|})_z -1) ( \N(D_4)_x - \N(C_4)_y)\,,
\end{align*}
where $x$ is the unique vertex of degree $3$ in the tadpole graph $D_4$, $y$ is any vertex of $C_4$, and $z$ is the unique vertex of degree greater than $2$ in $Q_{|V(H)|}$. We obtain
\begin{align*}
\N(R_{|V(G'_3)|}) - \N(G'_3) = -1 + (2^{|V(H)| -1} -1) (8 - 7) = 2^{|V(H)| -1} >0\,.
\end{align*}

\medskip
In either Subcase, we get
\begin{align*}
\N(R_{|V(G'_3)|}) > \N(G'_3) > \N(G^*_2) \quad \text{and} \quad \N(G_2) \geq \N(G_1) \geq \N(G_0)\,.
\end{align*}

\medskip
\item We are left with the situation where $3\leq p \leq q \leq |V(H)|\leq 4$ holds in the graph $G_2^*=G_0^*$. All such graphs $G_2^*$ are shown in Figure~\ref{fig:H4p4}.
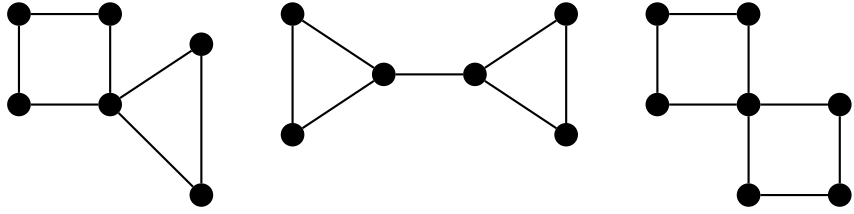
\begin{figure}[h]
	\begin{center}
		\begin{tikzpicture}
		[scale=0.4,inner sep=8mm, 
		vertex/.style={circle,thick,draw}, 
		thickedge/.style={line width=2pt}] 
		\node[vertex] (a1) at (0,0) [inner sep=1mm,fill=black] {};
		\node[vertex] (a2) at (3,0) [inner sep=1mm,fill=black] {};  
		\node[vertex] (a3) at (0,-3) [inner sep=1mm,fill=black] {};   
		\node[vertex] (a4) at (3,-3) [inner sep=1mm,fill=black] {};   
		\node[vertex] (a5) at (6,-1) [inner sep=1mm,fill=black] {};   
		\node[vertex] (a6) at (6,-6) [inner sep=1mm,fill=black] {};   
\draw[thick,black] (a1)--(a2) (a1)--(a3) (a2)--(a4) (a3)--(a4) (a4)--(a5) (a4)--(a6) (a5)--(a6);     
\node[vertex] (b1) at (9,0) [inner sep=1mm,fill=black] {};
		\node[vertex] (b2) at (9,-4) [inner sep=1mm,fill=black] {};  
		\node[vertex] (b3) at (12,-2) [inner sep=1mm,fill=black] {}; 
		\node[vertex] (b4) at (15,-2) [inner sep=1mm,fill=black] {};
		\node[vertex] (b5) at (18,0) [inner sep=1mm,fill=black] {};   
			\node[vertex] (b6) at (18,-4) [inner sep=1mm,fill=black] {};   
			\draw[thick,black] (b1)--(b2) (b1)--(b3) (b2)--(b3) (b3)--(b4) (b4)--(b5) (b4)--(b6) (b5)--(b6);     
\node[vertex] (c1) at (21,0) [inner sep=1mm,fill=black] {};
		\node[vertex] (c2) at (24,0) [inner sep=1mm,fill=black] {};  
		\node[vertex] (c3) at (21,-3) [inner sep=1mm,fill=black] {};   
		\node[vertex] (c4) at (24,-3) [inner sep=1mm,fill=black] {};   
		\node[vertex] (c7) at (27,-3) [inner sep=1mm,fill=black] {};   
		\node[vertex] (c5) at (24,-6) [inner sep=1mm,fill=black] {}; 
		\node[vertex] (c8) at (27,-6) [inner sep=1mm,fill=black] {}; 
		\draw[thick,black] (c1)--(c2) (c1)--(c3) (c2)--(c4) (c3)--(c4) (c4)--(c5) (c4)--(c7) (c5)--(c8) (c7)--(c8);     
\node[vertex] (d1) at (0,-8) [inner sep=1mm,fill=black] {};
		\node[vertex] (d2) at (0,-12) [inner sep=1mm,fill=black] {};  
		\node[vertex] (d3) at (2,-10) [inner sep=1mm,fill=black] {};   
		\node[vertex] (d4) at (4,-10) [inner sep=1mm,fill=black] {};   
		\node[vertex] (d5) at (6,-10) [inner sep=1mm,fill=black] {};   
		\node[vertex] (d6) at (8,-8) [inner sep=1mm,fill=black] {}; 
		\node[vertex] (d7) at (8,-12) [inner sep=1mm,fill=black] {}; 
		\draw[thick,black] (d1)--(d2) (d1)--(d3) (d2)--(d3) (d3)--(d4) (d4)--(d5) (d5)--(d6) (d5)--(d7) (d6)--(d7);     
\node[vertex] (e1) at (12,-8) [inner sep=1mm,fill=black] {};
		\node[vertex] (e2) at (12,-12) [inner sep=1mm,fill=black] {};  
		\node[vertex] (e3) at (14,-10) [inner sep=1mm,fill=black] {};   
		\node[vertex] (e4) at (16,-8) [inner sep=1mm,fill=black] {};   
		\node[vertex] (e5) at (16,-12) [inner sep=1mm,fill=black] {};   
		\draw[thick,black] (e1)--(e2) (e1)--(e3) (e2)--(e3) (e3)--(e4) (e3)--(e5) (e4)--(e5);     
		\node[vertex] (G) at (17.8,-11.3) [inner sep=1.mm,fill=white] {$R_5$};  
		\end{tikzpicture}
	\end{center}
	\caption{All graphs $G_2^*=G_0^*$ corresponding to the case where $3\leq p \leq q \leq |V(H)|\leq 4$ and $p\leq 4$. These graphs contain precisely $37,~30,~61,~53,~22$ connected sets, respectively.} \label{fig:H4p4}
\end{figure}

Recall that $G_2 \in \mathbb{B}_n$ can be constructed from $G_2^*=G_0^*$ by attaching (possibly none) a star $S_r$ rooted at only one vertex, say $v$ of $G_2^*$. In view of Lemma~\ref{Form:RnBn}, 
\begin{align*}
\N(R_5)=22\,, \quad \N(R_6)=39 > 37\,, \quad \N(R_7)=72 > 61\,.
\end{align*} 
In this case where $G_2^*$ is one of the graph shown in Figure~\ref{fig:H4p4}, we `replace' $G_2^*$ directly with $R_{|V(G_2^*)|}$ in $G_2$ to obtain the graph $R_n$. Since vertex $v$ is adjacent to all other vertices of $R_n$, Lemma~\ref{Lem:identify} combined with
\begin{align*}
\N(R_{|V(G_2^*)|}) > \N(G^*_2) \quad \text{and} \quad \N(G_2) \geq \N(G_1) \geq \N(G_0)
\end{align*}
give us $\N(R_n) > \N(G_2) \geq \N(G_1) \geq \N(G_0)$, and we are done, unless $G^*_2=R_5$. In this case, $G_2$ is indeed isomorphic to $R_n$ and we have $\N(B_n)> \N(R_n)=\N(G_2)$.
\end{itemize}

\medskip
{\sc Case~2}. Assume $G_2^*$ is of type~III (see Figure~\ref{fig:TypeIII}).

Then $G_2^*$ can be obtained from vertex disjoint paths $P_a$, $P_b$, $P_c$ by first identifying one endvertex of each of these paths, and then identifying the other endvertices of these paths. Thus $2\leq a\leq b \leq c$ and $b\geq 3$. Set $m:=|V(G^*_2)|$.

\medskip
{\sc Subcase~1}. First, assume $c\geq 5$.	

Let $w$ be a vertex of $G_2^*$ such that $v_1', v_1, w, v_2, v_2'$ are five consecutive vertices on the path $P_c$. We extract certain categories of vertex subsets of $G_2^*$, each of which contains $w$:
\begin{itemize}
\item Any nonempty subset of $V(G_2^*)\backslash \{v_1,v_2\}$; there are $2^{m-3}-1$ of them;
\item The union of $\{v_2\}$ and any nonempty subset of $V(G_2^*)\backslash \{v_1,v'_2\}$; there are $2^{m-4}-1$ of them;
\item The union of $\{v_1\}$ and any nonempty subset of $V(G_2^*)\backslash \{v_2,v'_1\}$; there are $2^{m-4}-1$ of them;
\item The union of $\{v_1,v_2\}$ and any nonempty subset of $V(G_2^*)\backslash \{v'_1,v'_2\}$; there are $2^{m-5}-1$ of them.
\end{itemize}
Each of the above subset of $V(G_2^*)$ contains $w$ and induces a disconnected graph. Moreover, all these subsets of $V(G^*_2)$ are pairwise distinct. Consequently,
\begin{align*}
\N(G_2^*)_w \leq & ~ 2^{m-1} - (2^{m-3}-1) - (2^{m-4}-1) - (2^{m-4}-1) - (2^{m-5}-1)\\
&=2^{m-2}+4-2^{m-5} \leq 2^{m-2} +2\,,
\end{align*}
since $m\geq c+1 >5$. On the other hand, $\N(G_2^* -w) < \N(Q_{m-1})$, since $G_2^* -w$ is a unicyclic graph on $m-1$ vertices, which is isomorphic neither to $Q_{m-1}$, nor to $C_5$, nor to the graph obtained from $C_4$ by attaching one pendant edge to a vertex of $C_4$.

Fix $z$, a vertex of degree $2$, belonging to a $C_3$ in $R_m$, and let $z'$ be the unique neigbhour of $z$ that lies on the other $C_3$ ($z'$ is adjacent to all other vertices of $R_m$). Then
\begin{align*}
\N(R_m)_z=\N(R_m)_{z,z'}+ \N(R_m-z')_z=2^{m-2}+2\,.
\end{align*} 
Hence
\begin{align*}
\N(R_m)&=\N(R_m)_z + \N(R_m-z) =\N(R_m)_z + \N(Q_{m-1}) =  2+2^{m-2} + \N(Q_{m-1}) \\
& >2+ 2^{m-2} + \N(G_2^* -w) \geq \N(G_2^*)_w + \N(G_2^* - w) = \N(G_2^*)\,.
\end{align*}
Recall that $G_2 \in \mathbb{B}_n$ can be derived from $G_2^*$ by attaching (possibly none) a star $S_{n-m+1}$ centered at only one vertex, say $v$ of $G_2^*$. Now `replace' $G_2^*$ with $R_m$ to obtain the graph $R_n$. We therefore infer from Lemma~\ref{Lem:identify} together with
\begin{align*}
\N(G_2) \geq \N(G_1) \geq \N(G_0)
\end{align*}
that $\N(R_n) > \N(G_2) \geq \N(G_1) \geq \N(G_0)$, and we are done with the case where $c\geq 5$.

\medskip
{\sc Subcase~2}. To finish the proof of the theorem, we need to rule out the case where $c\leq 4$ (thus $2\leq a \leq b \leq c \leq 4$ and $b\geq 3$). In this case $|V(G_2^*)| \leq 8$, see Figure~\ref{fig:2abc4}.
\begin{figure}[h]
	\begin{center}
		\begin{tikzpicture}
		[scale=0.4,inner sep=8mm, 
		vertex/.style={circle,thick,draw}, 
		thickedge/.style={line width=2pt}] 
		\node[vertex] (a1) at (0,0) [inner sep=1mm,fill=black] {};
		\node[vertex] (a2) at (3,3) [inner sep=1mm,fill=black] {};  
		\node[vertex] (a3) at (6,3) [inner sep=1mm,fill=black] {};   
		\node[vertex] (a4) at (8,1.5) [inner sep=1mm,fill=black] {};   
		\node[vertex] (a5) at (3,-3) [inner sep=1mm,fill=black] {};   
		\node[vertex] (a6) at (6,-3) [inner sep=1mm,fill=black] {};  
		\node[vertex] (a7) at (3,0) [inner sep=1mm,fill=black] {};   
		\node[vertex] (a8) at (6,0) [inner sep=1mm,fill=black] {};  
\draw[thick,black] (a1)--(a2) (a1)--(a5) (a1)--(a7) (a2)--(a3) (a3)--(a4) (a4)--(a6) (a4)--(a8) (a5)--(a6) (a7)--(a8);     
\node[vertex] (A) at (4.5,-5) [inner sep=1.mm,fill=white] {$E_8$};    
\node[vertex] (b1) at (12,0) [inner sep=1mm,fill=black] {};
		\node[vertex] (b2) at (15,3) [inner sep=1mm,fill=black] {};  
		\node[vertex] (b3) at (18,3) [inner sep=1mm,fill=black] {};   
		\node[vertex] (b4) at (20,1.5) [inner sep=1mm,fill=black] {};   
		\node[vertex] (b5) at (15,-3) [inner sep=1mm,fill=black] {};   
		\node[vertex] (b6) at (18,-3) [inner sep=1mm,fill=black] {};  
		\node[vertex] (b78) at (16.5,0) [inner sep=1mm,fill=black] {};   
\draw[thick,black] (b1)--(b2) (b1)--(b5) (b2)--(b3) (b3)--(b4) (b4)--(b6) (b5)--(b6) (b1)--(b78) (b78)--(b4); 
\node[vertex] (B) at (16.5,-5) [inner sep=1.mm,fill=white] {$E_7$};  
\node[vertex] (c1) at (24,0) [inner sep=1mm,fill=black] {};
		\node[vertex] (c2) at (27,3) [inner sep=1mm,fill=black] {};  
		\node[vertex] (c3) at (30,3) [inner sep=1mm,fill=black] {};   
		\node[vertex] (c4) at (32,1.5) [inner sep=1mm,fill=black] {};   
		\node[vertex] (c5) at (27,-3) [inner sep=1mm,fill=black] {};   
		\node[vertex] (c6) at (30,-3) [inner sep=1mm,fill=black] {};     
\draw[thick,black] (c1)--(c2) (c1)--(c5) (c2)--(c3) (c3)--(c4) (c4)--(c6) (c5)--(c6) (c1)--(c4);  
\node[vertex] (C) at (28.5,-5) [inner sep=1.mm,fill=white] {$E_{6,1}$}; 
	\node[vertex] (d1) at (0,-10) [inner sep=1mm,fill=black] {};
		\node[vertex] (d2) at (3,-7) [inner sep=1mm,fill=black] {};  
		\node[vertex] (d3) at (6,-7) [inner sep=1mm,fill=black] {};   
		\node[vertex] (d4) at (8,-8.5) [inner sep=1mm,fill=black] {};   
		\node[vertex] (d56) at (4.5,-13) [inner sep=1mm,fill=black] {};   
		\node[vertex] (d78) at (4.5,-10) [inner sep=1mm,fill=black] {};   
\draw[thick,black] (d1)--(d2) (d1)--(d56) (d1)--(d78) (d2)--(d3) (d3)--(d4) (d4)--(d56) (d4)--(d78);   
\node[vertex] (D) at (4.2,-15.2) [inner sep=1.mm,fill=white] {$E_{6,2}$};  
	\node[vertex] (e1) at (12,-10) [inner sep=1mm,fill=black] {};
		\node[vertex] (e2) at (15,-7) [inner sep=1mm,fill=black] {};  
		\node[vertex] (e3) at (18,-7) [inner sep=1mm,fill=black] {};   
		\node[vertex] (e4) at (20,-8.5) [inner sep=1mm,fill=black] {};   
		\node[vertex] (e56) at (16.5,-13) [inner sep=1mm,fill=black] {};   
\draw[thick,black] (e1)--(e2) (e1)--(e4) (e1)--(e56) (e2)--(e3) (e3)--(e4) (e4)--(e56);
\node[vertex] (E) at (16,-15.2) [inner sep=1.mm,fill=white] {$E_{5,1}$}; 
	\node[vertex] (f1) at (24,-10) [inner sep=1mm,fill=black] {};
		\node[vertex] (f23) at (27,-7) [inner sep=1mm,fill=black] {};     
		\node[vertex] (f4) at (32,-8.5) [inner sep=1mm,fill=black] {};   
		\node[vertex] (f56) at (28.5,-13) [inner sep=1mm,fill=black] {};   
		\node[vertex] (f78) at (28.5,-10) [inner sep=1mm,fill=black] {};   
\draw[thick,black] (f1)--(f23) (f1)--(f56) (f1)--(f78) (f23)--(f4) (f4)--(f56) (f4)--(f78); 
\node[vertex] (F) at (28.,-15.2) [inner sep=1.mm,fill=white] {$E_{5,2}$}; 
	\node[vertex] (g1) at (35,-10) [inner sep=1mm,fill=black] {};
		\node[vertex] (g23) at (39.5,-7) [inner sep=1mm,fill=black] {};  
		\node[vertex] (g4) at (41,-8.5) [inner sep=1mm,fill=black] {};   
		\node[vertex] (g56) at (39.5,-13) [inner sep=1mm,fill=black] {};   
\draw[thick,black] (g1)--(g23) (g1)--(g4) (g1)--(g56) (g23)--(g4) (g4)--(g56);\node[vertex] (G) at (38.5,-15.2) [inner sep=1.mm,fill=white] {$A_{4}$}; 
		\end{tikzpicture}
	\end{center}
	\caption{All graphs corresponding to $2\leq a \leq b \leq c \leq 4$ and $b\geq 3$ in $G_2^*$ of type~III.} \label{fig:2abc4}
	\end{figure}
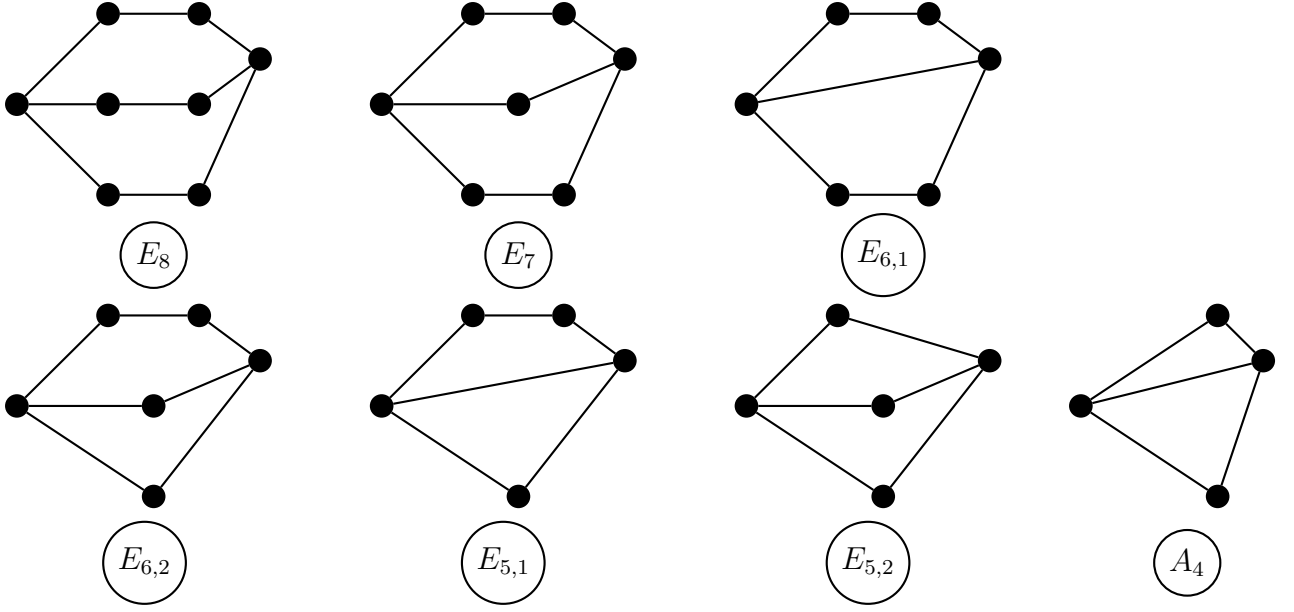

\medskip
We have computed both the number of connected sets and the maximum number of vertex-containing connected sets of each possibility from Figure~\ref{fig:2abc4}, and found that the upper bounds for the numbers of vertex-containing connected sets are attained precisely by any vertex of degree $3$ in the corresponding graph (see Appendix~\ref{Append}).

Recall that $n\geq 8$ and that $G_2 \in \mathbb{B}_n$ consists of
$$G_2^*\in \{A_4,E_{5,2}, E_{5,1}, E_{6,2}, E_{6,1}, E_{7}, E_{8}\}$$ and pendant edges, all of which are attached to the same vertex $v$ of $G_2^*$. With $m=|V(G_2^*)|$ and $u$ the center of $S_{n-m+1}$, we employ Lemma~\ref{Lem:identify} to obtain
\begin{align*}
\N(G_2)&=\N(G_2^*)+\N(S_{n-m+1}) -1 + (\N(G_2^*)_v-1)(\N(S_{n-m+1})_u - 1)\\
&=\N(G_2^*)+ (2^{n-m}+n-m) -1 + (\N(G_2^*)_v-1)(2^{n-m}- 1)\,.
\end{align*}
Hence, the following hold:
\begin{align*}
\N(G_2) \leq ~~& 26 + (2^{n-5}+n-5) -1 + (15 -1)(2^{n-5}- 1)=15 \cdot 2^{n-5}+6+n\\
&= 2^{n-1} -2^{n-5}+6+n < 2^{n-1} +1+n = \N(R_n)
\end{align*}
for $G_2^* \in \{E_{5,2},E_{5,1}\}$, and
\begin{align*}
\N(G_2)  \leq & ~~ 42 + (2^{n-6}+n-6) -1 + (25 -1)(2^{n-6}- 1)=25 \cdot 2^{n-6}+11+n\\
&= 2^{n-1} -7 \cdot 2^{n-6}+11+n < 2^{n-1} +1+n = \N(R_n)
\end{align*}
for $G_2^* \in \{E_{6,2},E_{6,1}\}$, and
\begin{align*}
\N(G_2)  \leq & ~~ 66 + (2^{n-7}+n-7) -1 + (41 -1)(2^{n-7}- 1)=41 \cdot 2^{n-7}+18+n\\
&= 2^{n-1} -23 \cdot 2^{n-7}+18+n < 2^{n-1} +1+n = \N(R_n)
\end{align*}
for $G_2^* =E_7$, and
\begin{align*}
\N(G_2)  \leq & ~~ 100 + (2^{n-8}+n-8) -1 + (64 -1)(2^{n-8}- 1)=64 \cdot 2^{n-8}+28+n\\
&= 2^{n-1} -64 \cdot 2^{n-8}+28+n < 2^{n-1} +1+n = \N(R_n)
\end{align*}
for $G_2^* =E_8$. In all these cases, we get $\N(G_2) < \N(R_n) =\N(B_n)-1 < \N(B_n)$. Recall that $B_n$ contains only one universal vertex (a vertex adjacent to all other vertices). Since the case where $G_2^* =A_4$ defines the graph $B_n$ when $v$ is of degree $3$ in $A_4$, we are done with the proof of the theorem.

In particular, $R_n$ is a second-largest $n$-vertex bicyclic graph with respect to the number of connected sets.
\end{proof}

\section{Case $n\in \{4, 5, 6, 7\}$}\label{Section:Case4TO7}

It was mentionned here at several occasions that $\mathbb{B}_4=\{A_4\}$ and $\N(A_4)=14$. On the other hand, we showed (see Figure~\ref{fig:All5}) that $E_{5,2}$ is the unique element of $\mathbb{B}_5$ with the maximum number of connected sets, and that $\N(E_{5,2})=26$.

In the proof of Theorem~\ref{Thm:MainOnMax}, the only place where it was necessary to invoke the assumption $n\geq 8$ is the situation where $G^*_2$ is of type~III. In fact, it was shown that if $n<8$ then $G^*_2$  cannot be of type~I or~II. Therefore, any $n\in \{6,7\}$-vertex bicyclic graph $G$ that realises the maximum number of connected sets must consist of its unique subgraph $G^*$ and a star centered at a single vertex of $G^*$, and $G^*$ must only be of type~III. We can therefore derive the following graph candidates $G$ on $6$ or $7$ vertices for maximality with respect to the number of connected sets :
\begin{enumerate}
\item The quantity $\N(A_4)_t$ is maximum if and only if vertex $t$ is of degree $3$ in $A_4$. Thus $G$ consists of $A_4$ with two or three pendant edges attached at a same vertex of degree $3$ in $A_4$ (one of them is precisely the graph $B_7$). 
\item The graphs $E_{5,1}$ and $E_{5,2}$ (see Figure~\ref{fig:2abc4}) are the only $5$-vertex graphs of type~III, while both $\N(E_{5,1})_t$ and $\N(E_{5,2})_t$ reach their maximum values if and only if vertex $t$ is of degree $3$. Thus $G$ can be obtained by attaching one or two pendant edges at a same vertex of degree $3$ in $E_{5,1}$ or $E_{5,2}$.
\item There are only three $6$-vertex graphs of type~III, namely $E_{6,1},~E_{6,2}$ (see Figure~\ref{fig:2abc4}), and $E_{6,3}$ shown in Figure~\ref{fig:PotentialMaxn=6}. It is easy to check that $\N(E_{6,3})_t$ is maximum if and only if vertex $t$ is of degree $3$ in $E_{6,3}$. Thus $G$ consists of $E_6 \in \{E_{6,1}, E_{6,2}, E_{6,3}\}$ with at most pendant edge attached at a vertex of degree $3$ in $E_6$.
\item The only $7$-vertex graphs of type~III (four in total) appear on the last row in Figure~\ref{fig:PotentialMaxn=7}.
\end{enumerate}
Overall, the search for elements of $\mathbb{B}_6$ or $\mathbb{B}_7$ with the maximum number of connected sets reduces to only $6$ or $10$ possible candidates, respectively. All such potential candidates are enumerated in Figures~\ref{fig:PotentialMaxn=6} and~\ref{fig:PotentialMaxn=7}, and their numbers of connected sets can easily be computed. For example, by decomposing these graphs into smaller ones (as in Appendix~\ref{Append}), or writing a simple computer program for the computations. 
\begin{figure}[h]
	\begin{center}
		\begin{tikzpicture}
		[scale=0.4,inner sep=0.6mm, 
		vertex/.style={circle,thick,draw}, 
		thickedge/.style={line width=2pt}] 
		\node[vertex] (a1) at (0,0) [fill=black] {};  
		\node[vertex] (a2) at (2,0) [fill=black] {};  
		\node[vertex] (a3) at (0,-2) [fill=black] {}; 
		\node[vertex] (a4) at (2,-2) [fill=black] {};  
		\node[vertex] (a5) at (4,-1) [fill=black] {}; 
		\node[vertex] (a6) at (4,-3) [fill=black] {}; 
		\draw[thick,black] (a1)--(a2) (a1)--(a3) (a1)--(a4) (a2)--(a4) (a3)--(a4) (a4)--(a5) (a4)--(a6); 
		[scale=0.4,inner sep=0.6mm, 
		vertex/.style={circle,thick,draw}, 
		thickedge/.style={line width=2pt}] 
		\node[vertex] (b1) at (8,0) [fill=black] {};  
		\node[vertex] (b2) at (10,0) [fill=black] {};  
		\node[vertex] (b3) at (6,-2) [fill=black] {}; 
		\node[vertex] (b4) at (12,-2) [fill=black] {};  
		\node[vertex] (b5) at (10,-4) [fill=black] {}; 
		\node[vertex] (b6) at (14,-2) [fill=black] {}; 
		\draw[thick,black] (b1)--(b2) (b1)--(b3) (b2)--(b4) (b3)--(b4) (b3)--(b5) (b4)--(b5) (b4)--(b6);    
		[scale=0.4,inner sep=0.6mm, 
		vertex/.style={circle,thick,draw}, 
		thickedge/.style={line width=2pt}] 
		\node[vertex] (c1) at (20,0) [fill=black] {};  
		\node[vertex] (c2) at (18,-2) [fill=black] {};  
		\node[vertex] (c3) at (22,-2) [fill=black] {}; 
		\node[vertex] (c4) at (20,-2) [fill=black] {};  
		\node[vertex] (c5) at (20,-4) [fill=black] {}; 
		\node[vertex] (c6) at (24,-2) [fill=black] {}; 
		\draw[thick,black] (c1)--(c2) (c1)--(c3) (c2)--(c4) (c2)--(c5) (c3)--(c4) (c3)--(c5) (c3)--(c6);      
		\node[vertex] (G) at (21,-5.7) [inner sep=1.mm,fill=white] {$E_{6,4}$}; 
		[scale=0.4,inner sep=0.6mm, 
		vertex/.style={circle,thick,draw}, 
		thickedge/.style={line width=2pt}] 
		\node[vertex] (d1) at (30,0) [fill=black] {};  
		\node[vertex] (d2) at (28,-2) [fill=black] {};  
		\node[vertex] (d3) at (32,-2) [fill=black] {}; 
		\node[vertex] (d4) at (28,-4) [fill=black] {};  
		\node[vertex] (d5) at (30,-6) [fill=black] {}; 
		\node[vertex] (d6) at (32,-4) [fill=black] {}; 
		\draw[thick,black] (d1)--(d2) (d1)--(d3) (d2)--(d3) (d2)--(d4) (d3)--(d6) (d4)--(d5) (d5)--(d6);      
		\node[vertex] (G) at (27.5,-6) [inner sep=1.mm,fill=white] {$E_{6,3}$}; 
		[scale=0.4,inner sep=0.6mm, 
		vertex/.style={circle,thick,draw}, 
		thickedge/.style={line width=2pt}] 
		\node[vertex] (e1) at (2,-6) [fill=black] {};  
		\node[vertex] (e2) at (0,-8) [fill=black] {};  
		\node[vertex] (e3) at (4,-6) [fill=black] {}; 
		\node[vertex] (e4) at (6,-8) [fill=black] {};  
		\node[vertex] (e5) at (0,-10) [fill=black] {}; 
		\node[vertex] (e6) at (4,-10) [fill=black] {}; 
		\draw[thick,black] (e1)--(e2) (e1)--(e3) (e2)--(e4) (e2)--(e5) (e3)--(e4) (e4)--(e6) (e5)--(e6);     
		\node[vertex] (G) at (6.4,-10.4) [inner sep=1.mm,fill=white] {$E_{6,1}$}; 
		[scale=0.4,inner sep=0.6mm, 
		vertex/.style={circle,thick,draw}, 
		thickedge/.style={line width=2pt}] 
		\node[vertex] (f1) at (12,-6) [fill=black] {};  
		\node[vertex] (f2) at (10,-8) [fill=black] {};  
		\node[vertex] (f3) at (14,-8) [fill=black] {}; 
		\node[vertex] (f4) at (12,-8) [fill=black] {};  
		\node[vertex] (f5) at (10,-10) [fill=black] {}; 
		\node[vertex] (f6) at (14,-10) [fill=black] {}; 
		\draw[thick,black] (f1)--(f2) (f1)--(f3) (f2)--(f4) (f2)--(f5) (f3)--(f4) (f3)--(f6) (f5)--(f6); 
		\node[vertex] (G) at (16,-10.) [inner sep=1.mm,fill=white] {$E_{6,2}$};     
		\end{tikzpicture}
	\end{center}
	\caption{All $6$-vertex bicyclic graphs that are candidates for the maximum number of connected sets.} \label{fig:PotentialMaxn=6}
\end{figure}

\begin{figure}[h]
	\begin{center}
		\begin{tikzpicture}
		[scale=0.4,inner sep=0.6mm, 
		vertex/.style={circle,thick,draw}, 
		thickedge/.style={line width=2pt}] 
		\node[vertex] (a1) at (0,0) [fill=black] {};  
		\node[vertex] (a2) at (2,0) [fill=black] {};  
		\node[vertex] (a3) at (0,-2) [fill=black] {}; 
		\node[vertex] (a4) at (2,-2) [fill=black] {};  
		\node[vertex] (a5) at (4,-1) [fill=black] {}; 
		\node[vertex] (a6) at (4,-3) [fill=black] {}; 
		\node[vertex] (a7) at (4,-2) [fill=black] {}; 
		\draw[thick,black] (a1)--(a2) (a1)--(a3) (a1)--(a4) (a2)--(a4) (a3)--(a4) (a4)--(a5) (a4)--(a6) (a4)--(a7); 
		\node[vertex] (G) at (2.3,-4) [inner sep=1.mm,fill=white] {$B_7$}; 
		[scale=0.4,inner sep=0.6mm, 
		vertex/.style={circle,thick,draw}, 
		thickedge/.style={line width=2pt}] 
		\node[vertex] (b1) at (8,0) [fill=black] {};  
		\node[vertex] (b2) at (10,0) [fill=black] {};  
		\node[vertex] (b3) at (6,-2) [fill=black] {}; 
		\node[vertex] (b4) at (12,-2) [fill=black] {};  
		\node[vertex] (b5) at (10,-4) [fill=black] {}; 
		\node[vertex] (b6) at (14,-2) [fill=black] {}; 
		\node[vertex] (b7) at (14,-1) [fill=black] {}; 
		\draw[thick,black] (b1)--(b2) (b1)--(b3) (b2)--(b4) (b3)--(b4) (b3)--(b5) (b4)--(b5) (b4)--(b6) (b4)--(b7);    
		[scale=0.4,inner sep=0.6mm, 
		vertex/.style={circle,thick,draw}, 
		thickedge/.style={line width=2pt}] 
		\node[vertex] (c1) at (20,0) [fill=black] {};  
		\node[vertex] (c2) at (18,-2) [fill=black] {};  
		\node[vertex] (c3) at (22,-2) [fill=black] {}; 
		\node[vertex] (c4) at (20,-2) [fill=black] {};  
		\node[vertex] (c5) at (20,-4) [fill=black] {}; 
		\node[vertex] (c6) at (24,-2) [fill=black] {}; 
		\node[vertex] (c7) at (24,-1) [fill=black] {}; 
		\draw[thick,black] (c1)--(c2) (c1)--(c3) (c2)--(c4) (c2)--(c5) (c3)--(c4) (c3)--(c5) (c3)--(c6)  (c3)--(c7);   
		\node[vertex] (G) at (22,-4.5) [inner sep=1.mm,fill=white] {$E^*_7$};    
%
		    
		[scale=0.4,inner sep=0.6mm, 
		vertex/.style={circle,thick,draw}, 
		thickedge/.style={line width=2pt}] 
		\node[vertex] (dp1) at (9,-7) [fill=black] {};  
		\node[vertex] (dp2) at (7,-9) [fill=black] {};  
		\node[vertex] (dp3) at (11,-9) [fill=black] {}; 
		\node[vertex] (dp4) at (7,-11) [fill=black] {};  
		\node[vertex] (dp5) at (9,-13) [fill=black] {}; 
		\node[vertex] (dp6) at (11,-11) [fill=black] {}; 
		\node[vertex] (dp7) at (13,-9) [fill=black] {};  
		\draw[thick,black] (dp1)--(dp2) (dp1)--(dp3) (dp2)--(dp3) (dp2)--(dp4) (dp3)--(dp6) (dp4)--(dp5) (dp5)--(dp6) (dp3)--(dp7);      
		[scale=0.4,inner sep=0.6mm, 
		vertex/.style={circle,thick,draw}, 
		thickedge/.style={line width=2pt}] 
		\node[vertex] (e1) at (2,-15) [fill=black] {};  
		\node[vertex] (e2) at (0,-17) [fill=black] {};  
		\node[vertex] (e3) at (4,-15) [fill=black] {}; 
		\node[vertex] (e4) at (6,-17) [fill=black] {};  
		\node[vertex] (e5) at (0,-19) [fill=black] {}; 
		\node[vertex] (e6) at (4,-19) [fill=black] {}; 
		\node[vertex] (e7) at (8,-17) [fill=black] {};  
		\draw[thick,black] (e1)--(e2) (e1)--(e3) (e2)--(e4) (e2)--(e5) (e3)--(e4) (e4)--(e6) (e5)--(e6) (e4)--(e7); 
		[scale=0.4,inner sep=0.6mm, 
		vertex/.style={circle,thick,draw}, 
		thickedge/.style={line width=2pt}] 
		\node[vertex] (f1) at (12,-15) [fill=black] {};  
		\node[vertex] (f2) at (10,-17) [fill=black] {};  
		\node[vertex] (f3) at (14,-17) [fill=black] {}; 
		\node[vertex] (f4) at (12,-17) [fill=black] {};  
		\node[vertex] (f5) at (10,-19) [fill=black] {}; 
		\node[vertex] (f6) at (14,-19) [fill=black] {};
		\node[vertex] (f7) at (16,-17) [fill=black] {};  
		\draw[thick,black] (f1)--(f2) (f1)--(f3) (f2)--(f4) (f2)--(f5) (f3)--(f4) (f3)--(f6) (f5)--(f6) (f3)--(f7);  
		[scale=0.4,inner sep=0.6mm, 
		vertex/.style={circle,thick,draw}, 
		thickedge/.style={line width=2pt}] 
		\node[vertex] (g1) at (4,-21) [fill=black] {};  
		\node[vertex] (g2) at (2,-23) [fill=black] {};  
		\node[vertex] (g3) at (6,-23) [fill=black] {}; 
		\node[vertex] (g4) at (2,-25) [fill=black] {};  
		\node[vertex] (g5) at (2,-27) [fill=black] {}; 
		\node[vertex] (g6) at (6,-27) [fill=black] {};
		\node[vertex] (g7) at (6,-25) [fill=black] {};
		\draw[thick,black] (g1)--(g2) (g1)--(g3) (g2)--(g3) (g2)--(g4) (g3)--(g7) (g4)--(g5) (g5)--(g6) (g6)--(g7); 
		[scale=0.4,inner sep=0.6mm, 
		vertex/.style={circle,thick,draw}, 
		thickedge/.style={line width=2pt}] 
		\node[vertex] (h1) at (9,-21) [fill=black] {};  
		\node[vertex] (hp1) at (13,-21) [fill=black] {};  
		\node[vertex] (h2) at (9,-23) [fill=black] {};  
		\node[vertex] (h3) at (13,-23) [fill=black] {}; 
		\node[vertex] (h4) at (9,-25) [fill=black] {};  
		\node[vertex] (h5) at (9,-27) [fill=black] {}; 
		\node[vertex] (h6) at (13,-27) [fill=black] {};
		\draw[thick,black] (h1)--(h2) (hp1)--(h3) (h2)--(h3) (h2)--(h4) (h4)--(h5) (h5)--(h6) (h3)--(h6) (h1)--(hp1);  
		[scale=0.4,inner sep=0.6mm, 
		vertex/.style={circle,thick,draw}, 
		thickedge/.style={line width=2pt}] 
		\node[vertex] (i1) at (18,-21) [fill=black] {}; 
		\node[vertex] (ip1) at (18,-23) [fill=black] {}; 
		\node[vertex] (i2) at (16,-23) [fill=black] {};  
		\node[vertex] (i3) at (20,-23) [fill=black] {}; 
		\node[vertex] (i4) at (16,-25) [fill=black] {};  
		\node[vertex] (i5) at (16,-27) [fill=black] {}; 
		\node[vertex] (i6) at (20,-27) [fill=black] {};
		\draw[thick,black] (i1)--(i2) (i1)--(i3) (i2)--(ip1) (i3)--(ip1) (i2)--(i4) (i4)--(i5) (i5)--(i6) (i3)--(i6);  
		[scale=0.4,inner sep=0.6mm, 
		vertex/.style={circle,thick,draw}, 
		thickedge/.style={line width=2pt}] 
		\node[vertex] (j1) at (25,-21) [fill=black] {}; 
		\node[vertex] (jp1) at (24,-25) [fill=black] {}; 
		\node[vertex] (jp2) at (26,-25) [fill=black] {}; 
		\node[vertex] (j2) at (23,-23) [fill=black] {};  
		\node[vertex] (j3) at (27,-23) [fill=black] {}; 
		\node[vertex] (j5) at (23,-27) [fill=black] {}; 
		\node[vertex] (j6) at (27,-27) [fill=black] {};
		\draw[thick,black] (j1)--(j2) (j1)--(j3) (j2)--(jp1) (j3)--(jp2) (jp1)--(jp2)  (j2)--(j5) (j5)--(j6) (j3)--(j6);      	
		\end{tikzpicture}
	\end{center}
	\caption{All $7$-vertex bicyclic graphs that are candidates for the maximum number of connected sets.} \label{fig:PotentialMaxn=7}
\end{figure}

\medskip
We found precisely two graphs, namely \(E_{6,2}\) and \(E_{6,4}\), that maximise the number of connected sets among the graphs in \(\mathbb{B}_6\). In particular, $
\N(E_{6,2})=\N(E_{6,4})=42$.

For \(n=7\), our computations show that \(B_7\) and \(E^*_7\) are the only graphs in \(\mathbb{B}_7\) that attain the maximum number of connected sets. More precisely, $
\N(B_7)=\N(E_7^*)=73$. We omit the details.

\section{Appendix} \label{Append}

In this section, we provide the enumeration of the connected sets and the vertex-containing connected sets for each of the Type~III configurations from Figure~\ref{fig:2abc4}. In particular, we obtain
\begin{align*}
&\N(A_4)=14 \,,~~ \N(A_4)_t \leq 8\,,\\
&\N(E_{5,2})=26 \,,~~ \N(E_{5,2})_t \leq 15\,,\\
&\N(E_{5,1})= 24\,,~~ \N(E_{5,1})_t \leq 14\,,\\
&\N(E_{6,2})= 42\,,~~ \N(E_{6,2})_t \leq 25\,,\\
&\N(E_{6,1})= 40\,,~~ \N(E_{6,1})_t \leq 25\,,\\
&\N(E_{7})=66 \,,~~ \N(E_{7})_t \leq 41\,,\\
&\N(E_{8})=100 \,,~~ \N(E_{8})_t \leq 64\,,
\end{align*}
where $t$ is any vertex of the corresponding graph. The stated upper bounds for the numbers of vertex-containing connected sets are attained precisely when $t$ is a vertex of degree $3$ in the corresponding graph.

We give below a bit of details of these computations.

\begin{itemize}
\item Denote by $u_2$ and $u'_2$ the two vertices of degree $3$ in $E_{6,2}$, and fix $u_1$ as one of their two common vertices. The graph $E_{6,2}-u'_2$ consists of $S_4$ centered at $u_2$ in which one pendant edge is subdivided only once. Thus, by symmetry, $\N(E_{6,2}-u'_2)_{u_1,u_2}=\N(E_{6,2}-u_2)_{u_1,u'_2}=6$. Furthermore, $\N(E_{6,2})_{u_1,u_2, u'_2}=2^3$ since every element of $V(E_{6,2})\backslash \{u_1, u_2, u'_2 \}$ is adjacent to adjacent to at least one element of $\{u_1, u_2, u'_2 \}$. Hence
\begin{align*}
\N(E_{6,2})_{u_1}&=\N(E_{6,2}- \{u_2,u'_2\})_{u_1} + \N(E_{6,2}-u'_2)_{u_1,u_2} + \N(E_{6,2}-u_2)_{u_1,u'_2}+\N(E_{6,2})_{u_1,u_2,u'_2}\\
&=1 + 6 + 6 + 8=21\,.
\end{align*}
Moreover, $\N(E_{6,2}-u_1)=\N(C_5)=5^2-5+1=21$. Therefore, $\N(E_{6,2})=\N(E_{6,2})_{u_1}+\N(E_{6,2}-u_1)=21+21=42$.

From the above description of $E_{6,2}-u_2$, we obtain $\N(E_{6,2}-u_2)=12+5=17$, depending on whether the connected sets involve $u'_2$ or not. In particular, $\N(E_{6,2})_{u_2}=\N(E_{6,2}) - \N(E_{6,2} -{u_2})=42-17=25$.

\medskip
Denote by $u_3$ the neighbour of $u_2$ in $E_{6,2}$, which is not a neighbour of $u'_2$. Then $E_{6,2}-u_3$ consists of $C_4$ and a leaf, say $a$ attached at the  vertex $u'_2$ of $C_4$. We have $\N(E_{6,2}-u_3)_a=1+\N(C_4)_{u'_2}=1+7=8$, which implies that $\N(E_{6,2}-u_3)= \N(E_{6,2}-u_3)_a + \N(E_{6,2}- \{u_3,a\})=8+\N(C_4)=8+13=21$ and that $\N(E_{6,2})_{u_3}=42-21=21$.
 
\item Consider the graph $E_{6,1}$. Fix $u_1$, a vertex of degree $3$, and let $u_2, u'_2, u'_1$ be its neighbours in $E_{6,1}$ with $u'_1$ of degree $3$. Then $\N(E_{6,1}-\{u_2,u'_2\})_{u_1}=5$ since $E_{6,1}-\{u_2,u'_2\}$ is the star $S_4$ in which $u_1$ is a leaf. Furthermore, $E_{6,1}- u'_2$ consists of $C_4$ and a leaf attached at the vertex $u'_1$ of $C_4$. Thus, by symmetry, $\N(E_{6,1}-u'_2)_{u_1,u_2} = \N(E_{6,1}- u_2)_{u_1,u'_2}=6$. It follows that
\begin{align*}
\N(E_{6,1})_{u_1}&=\N(E_{6,1}-\{u_2,u'_2\})_{u_1}+ \N(E_{6,1}- u'_2)_{u_1,u_2} + \N(E_{6,1}- u_2 )_{u_1,u'_2} + \N(E_{6,1})_{u_1,u_2,u'_2}\\
&=5+6+6+2^3=25\,.
\end{align*}
Since $E_{6,1} - u_1$ is isomorphic to $P_5$, we also have both $\N(E_{6,1} - u_1)=15$ and $\N(E_{6,1})=15+25=40$. With the vertex labels from $E_{6,2}$, the graphs $E_{6,1} - u_2$ and $E_{6,2} - u_3$ are isomorphic. Thus $\N(E_{6,1} - u_2)=21$ and $\N(E_{6,1})_{u_2}=40-21=19$.

\item Let us now focus on $E_7$. Denote by $u_2, u'_2$ the two vertices of degree $3$ and by $u_1$ their common neighbour. Then $E_7-u'_2$ consists of $S_4$ centered at $u_2$ with one subdivision of two edges of $S_4$, while $u_1u_2$ is a pendant edge in $E_7-u'_2$. Thus, by symmetry, $\N(E_7-u'_2)_{u_1,u_2}=\N(E_7-u_2)_{u_1,u'_2}=9$. Consequently, 
\begin{align*}
\N(E_7)_{u_1}&=\N(E_7- \{u_2, u'_2\})_{u_1} + \N(E_7-u'_2)_{u_1,u_2} + \N(E_7-u_2)_{u_1,u'_2} + \N(E_7)_{u_1,u_2,u'_2}\\
&=1+9+9+2^4=35\,.
\end{align*}
We note that $E_7-u_1$ is isomorphic to $C_6$. Thus $\N(E_7-u_1)=\N(C_6)=6^2-6+1=31$ and $\N(E_7)=31+35=66$. By the previous description of $E_7-u_2$, we obtain $\N(E_7-u_2)=18+7=25$, depending on whether the connected sets involve $u'_2$ or not. In particular, $\N(E_7)_{u_2}=66-25=41$. The graph $E_7-u_3$ consists of $C_5$ and a leaf, say $a$ attached at a vertex of $C_5$. It follows that $\N(E_7-u_3)_a=1+\N(C_5)_t=1+11=12$ and that $\N(E_7-u_3)= \N(E_7-u_3)_a + \N(E_7- \{u_3,a\})=12+\N(C_5)=12+21=33$. In particular, we get $\N(E_7)_{u_3}=66-33=33$.

\item Consider the graph $E_8$. Let $u_2, u'_2$ denote the two vertices of degree $3$, and fix a path $u_2u_1u'_1u'_2$. The graph $E_8-u_2$ consists of $S_4$ centered at $u'_2$ in which every edge is subdivided once. Vertex $u_1$ is a leaf adjacent to $u'_1$ in $E_8-u_2$, and one shows that $\N(E_8-u_2)_{u_1,u'_1}=1+9=10$. By symmetry, we get $\N(E_8-u'_2)_{u_1,u'_1,u_2}=9$. Furthermore, $E_8-u'_1$ is made of $C_6$ with leaf $u_1$ attached at vertex $u_2$ of $C_6$; this gives us $\N(E_8-u'_1)_{u_1,u_2} =\N(C_6)_{u_2}=16$. Hence
\begin{align*}
\N(E_8)_{u_1}&=\N(E_8-\{u'_1,u_2\})_{u_1} + \N(E_8-u_2)_{u_1,u'_1} + \N(E_8-u'_1)_{u_1,u_2} + \N(E_8)_{u_1,u'_1,u_2}\\
&= 1+ 10 + 16 + (2^4+9)=52\,,
\end{align*}
where $2^4+9$ is $\N(E_8)_{u_1,u'_1,u_2}$ split into containing $u'_2$ or not. Moreover,  by symmetry,
\begin{align*}
\N(E_8-u_1)&=\N(E_8-u'_1)=\N(E_8-u'_1)_{u_1}+\N(E_8-\{u'_1,u_1\})\\
&= (1+\N(C_6)_{u_2}) +\N(C_6) = (1+16)+ (6^2-6+1)=48\,.
\end{align*}
Therefore, we obtain $\N(E_8)=48+52=100$. By the above description of $E_8-u_2$, we have $\N(E_8-u_2)=\N(E_8-u_2)_{u'_2} + \N(E_8-\{u_2,u'_2\})=3^3 + (3 \times 3)=36$. It follows that $\N(E_8)_{u_2}=100-36=64$.

\item Basic decomposition arguments similar to the above can be used to compute $\N(E_{5,2})$, $\N(E_{5,1})$, $\N(A_4)$ as well as their numbers of vertex-containing connected sets. We can now skip the details.
\end{itemize}

\section*{Acknowledgment}

The authors are grateful to one anonymous reviewer for his/her valuable comments and suggestions which have improved the paper.

\medskip
\section*{Statements and Declarations}

The author did not receive support from any organization for the submitted work. The author has no competing interests to declare that are relevant to the content of this article.

\end{document}